\documentclass[11pt,letterpaper]{amsart}
\usepackage{geometry} 
\usepackage{mathtools}
\usepackage{graphicx}
\usepackage{paralist}
\usepackage[utf8]{inputenc}
\usepackage[T1]{fontenc}
\usepackage{csquotes}
\usepackage[all]{xy}
\usepackage{chngcntr}
\usepackage{amsthm}
\usepackage{amsmath}
\usepackage{amsfonts}
\usepackage{amssymb}
\usepackage{mathrsfs}
\usepackage{MnSymbol}
\usepackage{tikz-cd}
\usepackage{tikz, tikz-cd}
\usepackage{mathdots}
\usepackage{stmaryrd}
\usepackage{hyperref}
\usepackage{cleveref}
\usepackage{fancyhdr}
\usepackage{xcolor}

\newtheorem{theorem}{Theorem}
\newtheorem{thm}{Theorem}[section]
\newtheorem{lem}[thm]{Lemma}
\newtheorem{prop}[thm]{Proposition}
\newtheorem{cor}[thm]{Corollary}
\theoremstyle{remark}
\newtheorem*{remark}{Remark}

\newtheorem*{thm*}{Theorem}

\theoremstyle{definition}
\newtheorem{definition}[thm]{Definition}
\newtheorem*{question}{Question}

\newtheorem{example}[thm]{Example}

\counterwithin{equation}{section}

\newtheorem{rem}[thm]{Remark}

\newtheorem{fact}[thm]{Fact}

\newcommand{\bbN}{\mathbb{N}}
\newcommand{\bbZ}{\mathbb{Z}}
\newcommand{\bbQ}{\mathbb{Q}}
\newcommand{\bbR}{\mathbb{R}}

\newcommand{\bbH}{\mathbb{H}}

\DeclareMathOperator{\diam}{diam}
\DeclareMathOperator{\vol}{vol}
\DeclareMathOperator{\inj}{inj}
\DeclareMathOperator{\Ric}{Ric}

\DeclareMathOperator{\GL}{GL}

\DeclareMathOperator{\id}{id}
\DeclareMathOperator{\tr}{tr}
\DeclareMathOperator{\Sym}{Sym}

\begin{document}
\title{Negatively curved Einstein metrics on Gromov--Thurston manifolds}
\author{Ursula Hamenst\"adt and Frieder J\"ackel}
\thanks{Partially supported by the DFG Schwerpunktprogramm "Geometry at infinity"}
\date{January 14, 2024}

\begin{abstract}
For every $n\geq 4$ we construct infinitely many mututally not homotopic closed
manifolds of dimension $n$ which admit a negatively curved Einstein metric 
but no locally symmetric metric.
 \end{abstract}

\maketitle

\section{Introduction}

As a consequence of the solution to the geometrization conjecture by Perelmann, 
any closed manifold of dimension three which admits a negatively curved metric
also admits a hyperbolic metric, and for surfaces, the corresponding statement is
a classical consequence of the uniformization theorem. This statement is not 
true any more for closed negatively curved manifolds of dimension at least four.

Indeed, Gromov and Thurston \cite{GT87} constructed for each dimension 
$n\geq 4$ and every $\epsilon >0$ a closed manifold $X$ of dimension $n$ which 
admits a metric with curvature contained in the interval $[-1-\epsilon,-1+\epsilon]$
but which does \textit{not} admit a hyperbolic metric. These manifolds 
are cyclic coverings of standard arithmetic hyperbolic manifolds, branched along a 
null-homologous totally 
geodesic submanifold of codimension two. 
In the sequel we call such 
branched coverings 
\textit{Gromov--Thurston manifolds}. 

The proof of non-existence of hyperbolic
metrics on these manifolds is however indirect, that is, 
it it shown that among an infinite collection of candidate manifolds with
pinched curvature, 
only finitely many 
admit hyperbolic metrics. 
Much later, Ontaneda \cite{O20} gave a very general method for constructing closed 
Riemannian manifolds 
of any dimension $n\geq 4$ with arbitrarily pinched negative curvature.
Some of the examples he found have in addition some non-zero
Pontryagin numbers. 

It is a natural question whether these manifolds admit distinguished metrics. 
This was partially answered affirmatively
by Fine and Promoselli \cite{FP20} who constructed negatively curved Einstein 
metrics on an infinite family of Gromov-Thurston manifolds in dimension four. 
These metrics do not have constant curvature and therefore 
by the work of Besson, Courtois and Gallot \cite{BCG95}, see also the survey \cite{And10}, 
these manifolds are not homotopy equivalent to 
hyperbolic manifolds. 

The goal of this article is to extend this result to all dimensions. We show

\begin{theorem}\label{main thm}
For any $n\geq 4$ and any $\epsilon >0$ there exist infinitely many pairwise 
non-homeomorphic smooth closed manifolds $X$ of dimension $n$ with the
following properties.
\begin{enumerate}
\item $X$ admits a Riemannian metric with sectional curvature contained in 
$[-1-\epsilon,-1+\epsilon]$.
\item $X$ admits an Einstein metric with negative sectional curvature.
\item $X$ is \emph{not} homeomorphic to any closed locally symmetric space.
\end{enumerate}
\end{theorem}

The examples in the Theorem are Gromov--Thurston manifolds.  
For $n\geq 5$ they seem to be the first examples of negatively curved Einstein metrics
on manifolds which are \textit{not} locally symmetric.

Our construction builds on the ideas of Fine and Premoselli, however the examples
we find in dimension four are different from the examples in \cite{FP20}. 

The following question is motivated by the uniqueness of Einstein metrics on 
hyperbolic 4-manifolds \cite{BCG95}. 

\begin{question}
Is it true that an Einstein metric on any closed hyperbolic  manifold has constant curvature? 
Does every Gromov--Thurston manifold admit an Einstein metric?
\end{question}

The answer to the first 
question is negative if the Einstein metrics are allowed to have conical singularities (see \Cref{rem: conical Einstein metrics on M}). 

\subsection{Sketch of proof}

In this subsection we outline the rough strategy for the 
proof of \Cref{main thm} and point out the difference to the proof of Fine--Premoselli.

We start by constructing a particular sequence of closed hyperbolic 
manifolds of which we take the branched cover. 
Namely, using subgroup separability in arithmetic hyperbolic lattices of simplest  type, 
we construct for each $n \geq 4$ a sequence $(M_k)_{k \in \bbN}$ 
of closed hyperbolic manifolds that contain null-homologous closed totally geodesic codimension two submanifolds $\Sigma_k \subseteq M_k$ with at most two connected components and such that
\begin{equation}\label{eq: big normal injectivity radius}
	\lim_{k \to \infty}\frac{{\rm diam}(\Sigma_k)}{R_k^\nu} = 0,
\end{equation}
where $R^{\nu}_k$ is the normal injectivity radius of $\Sigma_k \subseteq M_k$ and, by a slight abuse of notation, ${\rm diam}(\Sigma_k)$ is the maximum of the diameters of the connected components of $\Sigma_k$. 

As $\Sigma_k\subseteq M_k$ is homologous to zero, for any fixed 
integer $d \geq 2$ there exists a cyclic $d$-fold covering $X_k$ of $M_k$, 
branched along $\Sigma_k$. Fine--Premoselli constructed an approximate Einstein metric $\bar{g}_k$ on $X_k$ by gluing together a model Einstein metric and the hyperbolic metric, where the gluing takes place in the region of distance $R_k$ away from $\Sigma_k$. 
We follow this strategy and choose as 
gluing parameter $R_k:=\frac{1}{2}R^{\nu}_k$. From (\ref{eq: big normal injectivity radius}) we then deduce that
\begin{equation}\label{eq: crucial L^2-estimate}
	\int_{X_k}|\Ric(\bar{g}_k)+(n-1)\bar{g}_k|^2(x) \, d\vol_{\bar{g}_k}(x) \xrightarrow{k \to \infty}0.
\end{equation}
From this estimate, we obtain the Einstein metric from an application of the inverse function theorem, using a uniform a priori estimate
for the so-called \emph{Einstein operator}. 

This leaves the question open whether the branched covering manifolds admit a locally symmetric 
metric. As our construction of the Einstein metrics uses a delicate volume estimate, to answer this 
question in the affirmative we can not rely on the indirect argument in \cite{GT87}. Moreover
the rigidity theorem in \cite{BCG95} only holds in dimension four. Instead we 
show in Section \ref{sec: not locally symmetric} 
a result of independent interest whose precise version is \Cref{prop: no hyperbolic metric}.
It states
that given a pair $(M,\Sigma)$ consisting of a closed hyperbolic $n$-manifold $M$ $(n\geq 4)$ 
and a codimension two
null-homologous embedded totally geodesic submanifold $\Sigma$ of $M$
of the form required for our construction,
among the cyclic covers of $M$ branched along $\Sigma$, 
at most one can be homeomorphic to a hyperbolic manifold. 

\subsection{Structure of the article}
The article is organized as follows. In \Cref{sec: Preliminaries} we review the necessary background information. Namely, \Cref{subsec: Einstein operator} introduces the Einstein operator. \Cref{subsec: C^0 estimate} explains how the De Giorgi--Nash--Moser estimates can be used to obtain $C^0$-estimates for the linearized Einstein operator. In \Cref{subsec: approximate Einstein} we recall the construction of the approximate Einstein metric on branched covers due to Fine--Premoselli. The algebraic results about arithmetic hyperbolic manifolds due to Bergeron and Bergeron--Haglund--Wise we use are contained in \Cref{subsec: subgroup separability}. These are then employed in \Cref{sec: good submanifolds} to construct the sequence of closed hyperbolic manifolds containing well-behaved codimension two submanifolds. In \Cref{subsec: Linearized Einstein is invertible} we show that the linearized Einstein operator is invertible. The existence of negatively curved Einstein metrics is then proved in \Cref{subsec: proof of main thm}. Finally, in \Cref{sec: not locally symmetric} we analyze Gromov Thurstion manifolds and, as 
an application, deduce that 
we can find such  manifolds in 
any dimension to which our construction applies and which do not admit any locally symmetric metric.

\bigskip\noindent
{\bf Acknowledgement:} U.H. thanks Alan Reid for pointing out the reference \cite{BHW11}, and Bena Tshishiku for pointing out
the reference \cite{CLW18}.

\section{Preliminaries}\label{sec: Preliminaries}

\subsection{The Einstein operator}\label{subsec: Einstein operator}
As mentioned in the introduction,
 we shall construct the Einstein metric by an application of the implicit function theorem for the so-called \emph{Einstein operator} (see \cite[Section I.1.C]{Biq00}, \cite[page 228]{And06} for more information).
This operator is defined as follows. 

Consider the operator $\Psi:g\to \Ric(g)+(n-1)g$ acting on smooth Riemannian metrics $g$ on the manifold $X$, 
where ${\rm Ric}$ denotes the Ricci tensor. 
As the diffeomorphism group ${\rm Diff}(X)$ of the manifold $X$ acts on metrics by pull-back and $\Psi$ is equivariant for this action,
the linearization of $\Psi$ is \textit{not} elliptic. To remedy this problem, for a given background metric \(\bar{g}\) one defines the \textit{Einstein operator} $\Phi_{\bar{g}}$ (in Bianchi gauge relative to $\bar{g}$) by 
\begin{equation}\label{eq: Def of Phi}
	\Phi_{\bar{g}}(g):=\Ric(g)+(n-1)g+\frac{1}{2}\mathcal{L}_{(\beta_{\bar{g}}(g))^\sharp}(g),
\end{equation}
where the musical isomorphism 
\(\sharp\) is with respect to the metric \(g\), and $\beta_{\bar{g}}$ is the \textit{Bianchi operator} of $\bar{g}$ acting on 
symmetric \((0,2)\)-tensors \(h\) by
\begin{equation}\label{bianchi}
	\beta_{\bar{g}}(h):=\delta_{\bar{g}}(h)+\frac{1}{2}d{\rm tr}_{\bar{g}}(h):=-\sum_{i=1}^n(\nabla_{e_i}h)(\cdot, e_i)+\frac{1}{2}d{\rm tr}_{\bar{g}}(h).
\end{equation}
The exact formula (\ref{eq: Def of Phi}) is not important. What does matter is that, using the formula for the linearization of \(\Ric\) (\cite[Proposition 2.3.7]{Top06}), one computes the linearization of \(\Phi_{\bar{g}}\) at \(\bar{g}\) to be
\begin{equation}\label{eq: linearisation of Einstein}
	(D\Phi_{\bar{g}})_{\bar{g}}(h)=\frac{1}{2}\Delta_L h +(n-1)h.
\end{equation}
Here $\Delta_L$ is the \textit{Lichnerowicz Laplacian} acting on symmetric $(0,2)$-tensors $h$ by
\[
    \Delta_Lh=\nabla^\ast \nabla h + \Ric(h),
\]
where $\nabla^\ast \nabla$ is the Connection Laplacian and $\Ric$ is the \textit{Weitzenböck curvature operator} given by
\(
	\Ric(h)(x,y)=h(\Ric(x),y)+h(x,\Ric(y))-2 \,{\rm tr}_g h(\cdot, R(\cdot,x)y)
\)
(see \cite[Section 9.3.2]{Pet16}).

\Cref{eq: linearisation of Einstein} shows that \((D\Phi_{\bar{g}})_{\bar{g}}\) is an elliptic operator. This opens up the possibility for an application of the implicit function theorem.

The main point is, as has been observed many times in the literature, that the Einstein operator can detect Einstein metrics. The following result
can for example be found in \cite[Lemma 2.1]{And06}.

\begin{lem}[Detecting Einstein metrics]\label{Zeros of Phi are Einstein}
Let \((X,\bar{g})\) be a complete Riemannian manifold, and let \(g\) be another metric on \(X\) so that 
\[
	\sup_{x \in X}|\beta_{\bar{g}}(g)|(x)<\infty  \quad \text{and} \quad \Ric(g) \leq \lambda g   \,\text{ for some } \lambda < 0,
\]
 where \(\beta_{\bar{g}}(\cdot)\) is the Bianchi operator of the background metric \(\bar{g}\). Denote by \(\Phi_{\bar{g}}\) the Einstein operator defined in (\ref{eq: Def of Phi}). Then
\begin{equation*}
 \Phi_{\bar{g}}(g)=0 \quad \text{if and only if } \quad g  \text{ solves the system} \quad \begin{cases} \Ric(g)=-(n-1)g \\
	\beta_{\bar{g}}(g)=0
	\end{cases}.
\end{equation*}
\end{lem}

\subsection{\texorpdfstring{$C^0$-}{estimate}}
\label{subsec: C^0 estimate}

To obtain $C^0$-estimates
for the linearization of the Einstein operator, we use once again a standard tool, the  
De Giorgi--Nash--Moser estimates on manifolds in the following form.

\begin{lem}[$C^0$-estimate]\label{lem: Nash-Moser} 
For all $n \in \bbN$, \(\alpha \in (0,1)\), \(\Lambda \geq 0\), and \(i_0>0\) there exist constants $\rho=\rho(n,\alpha,\Lambda,i_0) > 0$ and $C=C(n,\alpha,\Lambda,i_0)>0$ with the following property. 
Let \((X,g)\) be a Riemannian \(n\)-manifold satisfying 
\[
		|\sec(X,g)| \leq \Lambda  \quad \text{and} \quad \inj(X,g) \geq i_0.
\] 
For \(f \in C^0\big(\Sym^2(T^*X)\big)\) let \(h \in C^2\big({\rm  Sym}^2(T^*X)\big)\) be a solution of
\[
	\frac{1}{2}\Delta_L h+(n-1)h=f.
\]
Then it holds
\begin{equation}\label{C^0} 
	|h|(x) \leq C\Big( ||h||_{L^2(B(x,\rho))}+||f||_{C^0(B(x,\rho))}\Big)
\end{equation}
for all \(x \in M\).
\end{lem}

Here as customary, ${\rm sec}(X,g)$ denotes the sectional curvature of the metric $g$ and ${\rm inj}(X,g)$ the 
injectivity radius, and ${\rm Sym}^2(T^*X)$ is the bundle of symmetric $(0,2)$ tensors on $X$.

In the proof we will make use of a result by Jost--Karcher \cite[Satz 5.1]{JK82} or Anderson \cite[Main Lemma 2.2]{And90} that, under the geometric assumptions on $(X,g)$, around every point there exists harmonic charts of a priori size with good analytic control.

\begin{proof}The desired $C^0$-bound will follow from the De Giorgi–-Nash–-Moser estimates. The problem is that De Giorgi--Nash--Moser estimates only hold for scalar equations, but not for systems. To remedy this, we show that \(|h|\) satisfies an elliptic partial differential inequality.

We write $\Delta=\nabla^\ast \nabla =-\tr \nabla^2$ for the Connection Laplacian (on tensors and functions). Using \(\frac{1}{2}\Delta (|h|^2)=\langle \Delta h,h \rangle -|\nabla h|^2\) and \(f=\frac{1}{2}\Delta h+\frac{1}{2}\Ric(h)+(n-1)h\), we obtain
\[
	-\frac{1}{2}\Delta (|h|^2)=-2\langle f,h\rangle + \langle \Ric(h),h\rangle +2(n-1)|h|^2+ |\nabla h|^2.
\]
Note that \(|{\rm Ric}(h)| \leq C(n,\Lambda)|h|\) since \(|\sec(M)| \leq \Lambda\). Thus, together with the Cauchy-Schwarz inequality, the above equality implies
\begin{equation}\label{diff inequality 1}
	-\frac{1}{2}\Delta (|h|^2) + C(n,\Lambda)|h|^2 \geq -2|f| |h| + \vert \nabla h\vert^2.
\end{equation}

Suppose for the moment that \(h \neq 0\) everywhere. Then \(|h|\) is a nowhere vanishing \(C^2\) function. Observe
\[
	|\nabla(|h|)|\leq |\nabla h| \quad \text{and} \quad -\frac{1}{2}\Delta(|h|^2)=-|h|\Delta(|h|)+|\nabla(|h|)|^2.
\]
Combining this with inequality (\ref{diff inequality 1}) and dividing by \(|h|\) shows
\begin{equation}\label{differential inequality for |h|}
	-\Delta(|h|)+C(n,\Lambda)|h|\geq -2|f|.
\end{equation}

By \cite[Satz 5.1]{JK82} (also see \cite[Main Lemma 2.2]{And90}, \cite[page 230]{And06} and \cite[Proposition I.3.2]{Biq00}) there exist \(\rho=\rho(n,\alpha,\Lambda,i_0)>0\) and 
\(C=C(n,\alpha,\Lambda,i_0)\) with the following property. For all \(x \in X\) there exists a \textit{harmonic} chart \(\varphi:B(x,2\rho) \subseteq X \to \bbR^n\) centered at $x$ so that 
\begin{equation}\label{properties of admissable harmonic charts 1}
	e^{-Q}|v|_g \leq |(D\varphi)(v)|_{\rm eucl.} \leq e^Q |v|_g
\end{equation} 
for all \(v \in TB(x,2\rho)\), and for all \(i,j=1,\dots,n\)
\begin{equation}\label{properties of admissable harmonic charts 2}
	||g_{ij}^\varphi||_{C^{1,\alpha}} \leq C,
\end{equation}
where \(Q > 0\) is a very small fixed constant, and \(||\cdot||_{C^{2,\alpha}}\) is the usual Hölder norm of the coefficient functions in \(\varphi(B(p,2\rho) )\subseteq \bbR^n\).

Fix $x \in X$ and choose a harmonic chart $\varphi:B(x,2\rho) \to \bbR^n$ satisfying (\ref{properties of admissable harmonic charts 1}) and (\ref{properties of admissable harmonic charts 2}). In the local harmonic coordinates given by \(\varphi\), the differential inequality (\ref{differential inequality for |h|}) reads
\[
	g_{\varphi}^{ij}\partial_i\partial_j (|h| \circ \varphi^{-1})+C(n,\Lambda)(|h| \circ \varphi^{-1}) \geq -2(|f| \circ \varphi^{-1}) 
		\quad \text{in } \, \varphi\big(B(x_0,2\rho) \big) \subseteq \bbR^n.\]
Since the matrices \((g_{\varphi}^{ij})\) are uniformly elliptic by (\ref{properties of admissable harmonic charts 1}), the desired estimate (\ref{C^0}) follows from the classical De Giorgi–Nash–Moser estimates (see \cite[Theorem 8.17]{GT01}) provided that \(h \neq 0\) everywhere.

It remains to show that the assumption \(h \neq 0\) can be dropped. 
Note that (\ref{C^0}) is stable under \(C^2\)-convergence, that is,
if (\ref{C^0}) holds for a sequence of \(h_i\) and if \(h_i \to h\) in the \(C^2\)-topology, then (\ref{C^0}) also holds for \(h\). Therefore, it suffices to construct a sequence \(h_i\) converging to \(h\) in the \(C^2\)-topology so that \(h_i \neq 0\) everywhere.

An arbitrary \(h \in C^{2}\big( \Sym^2(T^*X)\big)\) can be approximated in the $C^2$-topology by symmetric 
$(0,2)$-tensors $h_i$ $(i\geq 1)$ which are transverse 
to the zero-section of $\Sym^2(T^*X)$. For reasons of dimension, such a section is disjoint from the zero-section, in other words, the tensors $h_i$ vanish nowhere. Therefore, the estimate (\ref{C^0}) holds for all \(h \in C^{2}\big( \Sym^2(T^*X)\big)\) and \(x \in X\).
\end{proof}

\subsection{The approximate Einstein metric of Fine--Premoselli}\label{subsec: approximate Einstein}

In this subsection we review the construction of the approximate Einstein metric on the branched cover of a hyperbolic manifold due to Fine--Premoselli. We refer the reader to \cite[Section 3]{FP20} for more information.

Let $M$ be a closed hyperbolic manifold of dimension $n \geq 4$, and let $\Sigma \subseteq M$ be a closed null-homologous totally geodesic codimension two embedded submanifold. Fix an integer $d \geq 2$ and denote by $p:X \to M$ the cyclic $d$-fold cover 
branched along $\Sigma$. We refer to \cite{FP20} and to Section \ref{sec: good submanifolds} for an explicit construction of 
such branched covers. 

Define $r:X \to \bbR$ by $r(x)=d_M(p(x),\Sigma)$ and $u=\cosh(r)$. Set
\[
    U_{\rm max}:=\cosh(R^\nu),
\]
where $R^\nu$ is the normal injectivity radius of $\Sigma \subseteq M$. The construction of the approximate Einstein metric also depends on a choice of gluing parameter $U_{\rm glue} < \frac{1}{2}U_{\rm max}$. We will later choose $U_{\rm glue}=(U_{\rm max})^{1/2}$, though this is irrelevant at the moment.

The following proposition summarizes all the necessary information about the approximate Einstein metric that will be used later.

\begin{prop}[The approximate Einstein metric]\label{prop: Properties of the approximate Einstein metric}
    There exists a smooth Riemannian metric $\bar{g}$ on the branched covering $X$ with the following properties:
    \begin{enumerate}[(i)]
        \item For all $m \in \bbN$ there exists a constant $C=C(m,n,d)$ such that
            \[
                ||\Ric(\bar{g})+(n-1)\bar{g}||_{C^m(X,\bar{g})} \leq C U_{\rm glue}^{-(n-1)};
            \]
        \item The tensor $\Ric(\bar{g})+(n-1)\bar{g}$ is supported in the region $\big\{\frac{1}{2}U_{\rm glue} < u < U_{\rm glue}\big\}$;
        \item There exists a constant $c=c(n,d) > 0$ such that $\sec(X,\bar{g}) \leq -c < 0$;
        \item For all $U < U_{\rm max}$ the volume of the region $\big\{\frac{1}{2}U < u < U\big\}$ is bounded from above by
            \[
                \vol_n\left(\left\{\frac{1}{2}U < u < U \right\},\bar{g}\right) 
                \leq
                C U^{n-1}\vol_{n-2}(\Sigma,g_{\rm hyp})
            \]
        for a constant $C=C(n,d)$.
    \end{enumerate}
\end{prop}

Points (i)-(iii) are contained in \cite[Proposition 3.1]{FP20}. Property (iii) requires that the gluing parameter
$U_{\rm glue}$ is larger than a constant depending on $n$ and $d$. As in \cite{FP20}, in our construction this
will always be the case. 
Point (iv) follows from the explicit construction, which we are now going to explain. 

Consider $\bbH^n$ and fix a totally geodesic copy $S \subseteq \bbH^n$ of $\bbH^{n-2}$. Then in exponential normal coordinates centered at $S$, the hyperbolic metric of $\bbH^n$ is given by
\[
    g_{\bbH^n}=dr^2+\sinh^2(r)d\theta^2+\cosh^2(r)g_S,
\]
where $g_S$ is the hyperbolic metric of $S$. Using the change of variables $u=\cosh(r)$, this becomes
\[
    g_{\bbH^n}=\frac{du^2}{u^2-1}+(u^2-1)d\theta^2+u^2g_S,
\]
which is defined for $(u,\theta) \in (1,\infty) \times S^1$.

Fine--Premoselli consider metrics of the form
\begin{equation}\label{eq: Ansatz for Einstein metric}
    g=\frac{du^2}{V(u)}+V(u)d\theta^2+u^2g_S,
\end{equation}
where $V$ is a positive smooth function. The following is \cite[Proposition 3.2]{FP20}.

\begin{lem}
The metric $g$ defined in (\ref{eq: Ansatz for Einstein metric}) solves $\Ric(g)+(n-1)g=0$ if and only if $V$ is of the form
\begin{equation}\label{eq: Einstein functions}
    V(u)=u^2-1+\frac{a}{u^{n-3}}
\end{equation}
for some $a \in \bbR$. 
\end{lem}

Let $g_a$ be the metric (\ref{eq: Ansatz for Einstein metric}) for the function $V=V_a$ given by (\ref{eq: Einstein functions}). Let $u_a$ denote the largest zero of the function $V_a$. If $u_a > 0$, $g_a$ is a smooth Riemannian metric for $u \in (u_a,\infty)$, but in general it will have a cone singularity along $S$ at $u=u_a$ with cone angle depending on $a$. The following summarizes \cite[Lemma 3.3]{FP20}.

\begin{lem}\label{lem: cone singularities of Ansatz}
There are explicit constants $a_{\rm max}=a_{\rm max}(n) > 0$ and $v=v(n)>0$ such that the following hold.
\begin{enumerate}[(i)]
    \item We have $u_a > 0$ if and only if $a \in (-\infty,a_{\rm max}]$ and the map $a \mapsto u_a$ is a decreasing homeomorphism $(-\infty,a_{\rm max}] \to [v,\infty)$;
    \item For each integer $d \geq 1$ there exists a unique $a=a(d) \in (-\infty,a_{\rm max})$ such that the cone angle of $g_a$ along $S$ at $u=u_a$ is $\frac{2\pi}{d}$.
    \item The sequence $(a(d))_{d \in \bbN}$ is strictly increasing with $a(1)=0$ and $a(d) \to a_{\rm max}$ as $d \to \infty$. 
\end{enumerate}
\end{lem}

Therefore, the metric $g_{a(d)}$ defines a global smooth metric on the cyclic $d$-fold branched cover of $\bbH^n$ along $S$. Of course, this is also true in $X$, the cyclic $d$-fold branched cover of $M$ along $\Sigma$, at least in a tubular neighbourhood of $\Sigma$. 

The approximate Einstein metric $\bar{g}$ in \Cref{prop: Properties of the approximate Einstein metric} is then obtained by interpolating between $g_{a(d)}$ (defined in a tubular neighbourhood of $\Sigma$) and $g_{\rm hyp}$ (defined on $X \setminus \Sigma$). Namely, $\bar{g}$ is as in (\ref{eq: Ansatz for Einstein metric}) for a function $V$ of the form
\[
    V=u^2-1+\frac{a}{u^{n-3}}\chi(u),
\]
where $\chi$ smooth cutoff function with $\chi=1$ in $\{u \leq \frac{1}{2}U_{\rm glue}\}$ and $\chi=0$ in $\{u \geq U_{\rm glue}\}$. We refer the reader to \cite[Section 3.2]{FP20} for further details.

The volume estimate in \Cref{prop: Properties of the approximate Einstein metric}(iv) follows easily because $\bar{g}$ is of the form (\ref{eq: Ansatz for Einstein metric}).

\subsection{Algebraic retraction and subgroup separability}\label{subsec: subgroup separability}

In this subsection we state the results about arithmetic hyperbolic manifolds and arithmetic groups from Bergeron--Haglund--Wise \cite{BHW11} 
and Bergeron \cite{Ber00} that are needed for the construction of good totally geodesic submanifolds of codimension two (see \Cref{prop: good submanifolds}).

Consider a $\mathbb{Q}$- algebraic group $\bf{G}$ such that the group of its real points is
the product, with finite intersection, of a compact group by the isometry group 
${\rm O}^+(n,1)$ of $\mathbb{H}^n$ for some
$n\geq 4$. We require that $\bf G$
comes by restriction of scalars from an orthogonal group over a totally real number field and that
$\bf{G}$ is anisotropic over $\mathbb{Q}$. 
The arithmetic group $\Gamma$ is the subgroup of $\bf{G}$ which is 
defined over the ring of integers in the totally real number field; it  
acts cocompactly on $\mathbb{H}^n$. The compact 
hyperbolic orbifold $\Gamma\backslash \mathbb{H}^n$ is 
called \emph{standard}, and $\Gamma$ is a called a standard arithmetic lattice or an arithmetic lattice of 
simplest type. 
A \emph{sufficiently deep} congruence subgroup $\Gamma^\prime$ of $\Gamma$ is known to be torsion free
and hence acts freely on $\mathbb{H}^n$. 
Following \cite{BHW11} we call the quotient $\Gamma^\prime \backslash \mathbb{H}^n$ a 
\emph{standard arithmetic hyperbolic manifold}.

A \emph{$\Gamma$-hyperplane} in $\mathbb{H}^n$ is a totally geodesic hyperplane $H\subset \mathbb{H}^n$ with the property
that ${\rm Stab}_\Gamma(H)$ acts cocompactly on $H$. 
If $\Gamma$ is an arithmetic group, then 
it is well-known that there exists a $\Gamma$-hyperplane in $\bbH^n$ 
if and only if $\Gamma$ is standard. Similarly, a \emph{$\Gamma$-subspace} is a totally geodesic 
subspace $\Sigma$ of $\mathbb{H}^n$ of arbitrary codimension
so that ${\rm Stab}_{\Gamma}(\Sigma)$ acts cocompactly on $\Sigma$.

\begin{definition}\label{def:subsep}
A subgroup $\Lambda$ of a group $\Gamma$ is called \emph{separable} if for any 
$\gamma\in \Gamma \setminus \Lambda$, there exists a finite index subgroup $\Gamma^\prime\leqslant \Gamma$ such that $\Lambda \leqslant \Gamma^\prime$ and $\gamma \notin \Gamma^\prime$.
\end{definition}

The following is a special case of a result of Bergeron (see \cite[Lemme principal]{Ber00} or \cite[Corollary 1.12]{BHW11}).

\begin{thm}[Subgroup Separability]\label{subgroupsep}
Let $M=\Gamma \backslash \bbH^n$ be a standard arithmetic hyperbolic manifold and $\Sigma$ a $\Gamma$-subspace. Then ${\rm Stab}_{\Gamma}(\Sigma)$ is separable in $\Gamma$.

\end{thm}

Note that if $\Gamma$ is a group as in \Cref{subgroupsep}, if $\Sigma$ is a $\Gamma$-subspace 
and if $\Gamma^\prime$ is a finite index subgroup of $\Gamma$, then 
${\rm Stab}_{\Gamma^\prime}(\Sigma)\leqslant \Gamma^\prime$ is separable. Indeed, if $\gamma\in \Gamma^\prime \setminus {\rm Stab}_{\Gamma^\prime}(\Sigma)$, then $\gamma\in \Gamma \setminus {\rm Stab}_\Gamma(\Sigma)$. Thus 
if $\Gamma^{\prime\prime}$ is a finite index subgroup of 
$\Gamma$ which contains ${\rm Stab}_{\Gamma}(\Sigma)$ but not $\gamma$, then 
$\Gamma^{\prime\prime}\cap \Gamma^\prime$ is a finite index subgroup of $\Gamma^\prime$
which contains ${\rm Stab}_{\Gamma^\prime}(\Sigma)$ but not $\gamma$.

The following summarizes the results of \cite{BHW11} that will be needed in the sequel.

\begin{thm}[Bergeron--Haglund--Wise]\label{thm: BHW retraction}
Let $M= {\Gamma} \backslash {\bbH^n}$ be a standard arithmetic hyperbolic manifold and $H \subseteq \bbH^n$ a $\Gamma$-hyperplane.
Then there exists a subgroup of finite index $\Gamma^\prime \leqslant \Gamma$ that retracts onto ${\rm Stab}_{\Gamma^\prime}(H)$, that is, there is a group homomorphism
\[
    {\rm retr}: \Gamma^\prime \to {\rm Stab}_{\Gamma^\prime}(H)
    \quad \text{such that} \quad
    {\rm retr}|_{{\rm Stab}_{\Gamma^\prime}(H)}={\rm id}_{{\rm Stab}_{\Gamma^\prime}(H)}.
\]
Moreover, ${\rm Stab}_{\Gamma^\prime}(H) \backslash H$ is a standard arithmetic hyperbolic manifold, and the natural map 
\[
    {\rm Stab}_{\Gamma^\prime}(H) \backslash H \longrightarrow \Gamma^\prime \backslash \bbH^n 
\]
is an embedding whose image agrees with the projection of $H \subseteq \bbH^n$ to $\Gamma^\prime \backslash \bbH^n$.
\end{thm}

The first half of \Cref{thm: BHW retraction} is \cite[Theorem 1.2]{BHW11}. So it remains to explain why the second part easily follows from the results of \cite{BHW11}.

\begin{proof}
By \cite[Theorem 1.2]{BHW11} there exists a torsion free congruence subgroup $\Gamma^\prime \leqslant \Gamma$ of finite index. Note that,
for any $\Gamma$-hyperplane $H\subseteq \mathbb{H}^n$, the stabilizer ${\rm Stab}_{\Gamma^\prime}(H)$ of 
$H$ in $\Gamma^\prime$ is a congruence subgroup of the arithmetic group 
${\rm Stab}_\Gamma(H)$. 

Appealing to \cite[Theorem 1.4]{BHW11} we may assume, after possibly passing to a further finite index congruence subgroup, that $\Lambda={\rm Stab}_{\Gamma^{\prime}}(H)$ is a virtual retract of 
$\Gamma^{\prime}$, that is, there exists a finite index subgroup
$Q \leqslant \Gamma^{\prime}$ containing $\Lambda$ and a homomorphism $Q\to \Lambda$ that is the identity when restricted to $\Lambda$.

Now $\Lambda\backslash H$ is a compact standard arithmetic manifold, and the natural map $\Lambda \backslash H\to Q\backslash \mathbb{H}^n$ induced by the inclusion $H \subseteq \bbH^n$ is an immersion. By \Cref{subgroupsep} and
the following remark,
the subgroup $\Lambda$ is separable in $Q$. 

If the immersion $\Lambda\backslash H \to Q \backslash \bbH^n$ is not an embedding, then the hyperplane $H$ is not precisely invariant
under $Q$, that is, there exist $\gamma\in Q \setminus \Lambda$ such that
\begin{equation}\label{eq: no embedding 1}
	\gamma(H)\cap H\neq \emptyset 
	\quad \text{but} \quad
	\gamma(H)\neq H.
\end{equation}
Then $\gamma(H)\cap H$ is the intersection of two totally geodesic hyperplanes
and hence it is a totally geodesic submanifold of $H$ of codimension one.
As the action of $\Lambda$ on $H$ is cocompact, there exists a compact fundamental
domain $D \subseteq H$ for the action of $\Lambda$ on $H$. For any 
$\gamma \in Q \setminus \Lambda$ satisfying (\ref{eq: no embedding 1}) there exists, by precomposing 
with a suitable element from $\Lambda$, an element $\gamma^\prime \in Q \setminus \Lambda$ such that
\begin{equation}\label{eq: no embedding 2}
	\gamma^\prime(H)\cap H\neq \emptyset 
	\quad \text{and} \quad
	\gamma^\prime (D) \cap  D \neq 0.
\end{equation}

Since the action of $Q$ on $\bbH^n$ is discrete and $D$ is compact, there exist only finitely many 
elements in $Q \setminus \Lambda$ satisfying (\ref{eq: no embedding 2}). 
Keeping in mind that $\Lambda$ is separated in $Q$, we can find a finite
index subgroup $Q^\prime \leqslant Q$ which contains $\Lambda$ but does not contain any of the elements satisfying (\ref{eq: no embedding 2}), and hence also no element satisfying (\ref{eq: no embedding 1}). Then the manifold $\Lambda\backslash H$ is embedded in $Q^\prime\backslash \mathbb{H}^n$. 
Furthermore, the restriction of the retraction $Q\to \Lambda$ to $Q^\prime$ defines a retraction 
$Q^\prime \to \Lambda$. This completes the proof.
\end{proof}

\section{Good totally geodesic submanifolds of codimension two}\label{sec: good submanifolds}

The goal of this section is to prove the following proposition.

\begin{prop}[Codimension two submanifolds]\label{prop: good submanifolds}
For each $n \geq 4$ and any standard arithmetic hyperbolic manifold $M$, there is 
a sequence of finite covers  
$(M_k)_{k \in \bbN}$ of $M$ containing closed embedded totally geodesic 
submanifolds $\Sigma_k \subset H_k\subset M_k$ with the following properties:
    \begin{enumerate}[(i)]
        \item The manifolds $\Sigma_k$ are all isometric, and they are of codimension 2.
        \item $\Sigma_k$ is null-homologous in the embedded connected hypersurface $H_k\subset M_k$;
        \item $\Sigma_k$ has at most two connected components;
        \item We have
        \[
            \lim_{k \to \infty}\frac{{\rm diam}(\Sigma_k)}{R_k^{\nu}}=0,
        \]
        where $R_k^\nu$ is the normal injectivity radius of $\Sigma_k \subseteq M_k$ and, by abuse of notation, $\diam(\Sigma_k)$ is the maximum of the diameters of the connected components of $\Sigma_k$.
    \end{enumerate}
\end{prop}

Before we come to the proof of \Cref{prop: good submanifolds}, we first show how it can be used to control the $L^2$-norm of the approximate Einstein metrics.

Namely, let $(M_k)_{k \in \bbN}$ and $\Sigma_k \subseteq M_k$ be as in \Cref{prop: good submanifolds}. Fix an integer $d \geq 2$, and denote by $X_k$ the cyclic $d$-fold cover of $M_k$ branched along $\Sigma_k$. Set 
\begin{equation}\label{eq: gluing parameters}
    U_{k;\rm max}:=\cosh(R_k^\nu)
   \quad \text{and} \quad
   U_{k;\rm glue}:=\left(U_{k;\rm max} \right)^{\frac{1}{2}}.
\end{equation}
Let $\bar{g}_k$ be the approximate Einstein metric on $X_k$ given by \Cref{prop: Properties of the approximate Einstein metric}. Then we easily obtain the following from \Cref{prop: Properties of the approximate Einstein metric} and \Cref{prop: good submanifolds}.

\begin{cor}[Small $L^2$-norm]\label{cor: small L2 norm}
For the $L^2$-norm of $\Ric(\bar{g}_k)+(n-1)\bar{g}_k$ we have
\[
    \int_{X_k} |\Ric(\bar{g}_k)+(n-1)\bar{g}_k|^2(x) \, d\vol_{\bar{g}_k}(x) \xrightarrow{k \to \infty}0.
\]
\end{cor}

This is the key estimate that will enable us to use a fairly simple 
perturbation argument for the construction of Einstein metrics. 

\begin{proof}[Proof of \Cref{cor: small L2 norm}]
    By \Cref{prop: Properties of the approximate Einstein metric}(i),(ii) the tensor $\Ric(\bar{g}_k)+(n-1)\bar{g}_k$ is supported in the region $\big\{\frac{1}{2}U_{k; \rm glue} < u < U_{k; \rm glue}\big\}$ and it is uniformly bounded from above by
    \begin{equation}\label{eq: C^0 bound of trace free ricci}
        ||\Ric(\bar{g}_k)+(n-1)\bar{g}_k||_{C^0(X_k,\bar{g}_k)} \leq CU_{k;\rm glue}^{-(n-1)}
    \end{equation}
    It follows from \Cref{prop: good submanifolds}(iii),(iv) and the definition (\ref{eq: gluing parameters}) of $U_{k; \rm glue}$ that for all $\varepsilon > 0$ and $k \geq k_0(\varepsilon)$ large enough we have
    \begin{equation*}
        \vol_{n-2}(\Sigma_k,g_{\rm hyp})\leq C\exp\big((n-3)\diam(\Sigma_k)\big) 
        \leq C\exp\left(\frac{1}{2}\varepsilon R_k^\nu \right)
        \leq CU_{k; \rm glue}^\varepsilon.
    \end{equation*}
    Together with the volume bound in \Cref{prop: Properties of the approximate Einstein metric}(iv) this implies
    \begin{equation}\label{eq: volume of gluing region}
        \vol_n\left(\left\{\frac{1}{2}U_{k; \rm glue} < u < U_{k; \rm glue} \right\},\bar{g}\right) 
            \leq
        C U_{k; \rm glue}^{n-1}\vol_{n-2}(\Sigma,g_{\rm hyp})
            \leq
        C U_{k; \rm glue}^{n-1+\varepsilon}.
    \end{equation}
    Combining (\ref{eq: C^0 bound of trace free ricci}) and (\ref{eq: volume of gluing region}) implies the desired estimate given that we choose $\varepsilon < n-1$.
\end{proof}

We now come to the proof of \Cref{prop: good submanifolds}.

\begin{proof}[Proof of \Cref{prop: good submanifolds}]
The proof of \Cref{prop: good submanifolds} is divided into three steps.

\smallskip\noindent
\textit{Step 1.} Let $M$ be a standard arithmetic hyperbolic manifold. 
Because of \Cref{thm: BHW retraction} and \cite[Theorem 1.4]{BHW11}, after passing to a finite cover, we may assume that
$M=\Gamma\backslash \mathbb{H}^n$ where $\Gamma$ is a torsion free cocompact lattice and that there exists a $\Gamma$-hyperplane $\tilde{H}\subseteq \mathbb{H}^n$ 
with the following properties:
\begin{enumerate}
\item ${\rm Stab}_{\Gamma}(\tilde{H})\backslash \tilde{H}$ is a standard arithmetic hyperbolic manifold which is 
embedded in $\Gamma\backslash \mathbb{H}^n$;
\item There is a retraction ${\rm retr}:\Gamma \to {\rm Stab}_{\Gamma}(\tilde{H})$;
\item Any geometrically finite subgroup of $\Gamma$ is a virtual retract. 
\end{enumerate}

Since ${\rm Stab}_{\Gamma}(\tilde{H})\backslash \tilde{H}$ is a standard arithmetic hyperbolic manifold, there
exists a ${\rm Stab}_{\Gamma}(\tilde{H})$-hyperplane $\tilde{\Sigma}\subseteq \tilde{H}$ and a subgroup 
$Q \leqslant {\rm Stab}_{\Gamma}(\tilde{H})$ of finite index such that 
\[
	\Sigma={\rm Stab}_{Q}(\tilde{\Sigma})\backslash \tilde{\Sigma}
\] 
is a standard arithmetic hyperbolic manifold embedded in 
$Q\backslash H$.

The preimage $\Gamma^\prime:={\rm retr}^{-1}(Q)$ of $Q \leqslant  {\rm Stab}_{\Gamma}(\tilde{H})$ under the retraction ${\rm retr}: \Gamma \to  {\rm Stab}_{\Gamma}(\tilde{H})$ is a finite index subgroup of $\Gamma$. Note that ${\rm Stab}_{\Gamma^\prime}(\tilde{H})=\Gamma^\prime \cap {\rm Stab}_{\Gamma}(\tilde{H})=Q$, and so the retraction of $\Gamma$ restricts to a retraction ${\rm retr}:\Gamma^\prime \to {\rm Stab}_{\Gamma^\prime}(\tilde{H})$. 
Moreover, any geometrically finite subgroup of $\Gamma^\prime$ is still a virtual retract of $\Gamma^\prime$.

Therefore, we have obtained a finite index subgroup $\Gamma^\prime \leqslant \Gamma$ such that ${\rm Stab}_{\Gamma^\prime}(\tilde{\Sigma}) \backslash \tilde{\Sigma}$, ${\rm Stab}_{\Gamma^\prime}(\tilde{H}) \backslash \tilde{H}$ and $\Gamma^\prime \backslash \bbH^n$ are all standard arithmetic hyperbolic manifolds that are smoothly embedded in each other and so that (2),(3) above still hold for $\Gamma^\prime$.

For ease of notation we will from now on replace $\Gamma^\prime$ by $\Gamma$ in our notations and put 
$M=\Gamma\backslash \mathbb{H}^n$. Furthermore we put 
\[
    \Gamma_{\Sigma}:={\rm Stab}_{\Gamma}(\tilde{\Sigma}),
    \quad
    \Sigma:=\Gamma_{\Sigma} \backslash \tilde{\Sigma},
    \quad
    \Gamma_H:={\rm Stab}_{\Gamma}(\tilde{H}),
    \quad   H:=\Gamma_H\backslash \tilde{H}.
\]
This notation is slightly different than in \Cref{subsec: subgroup separability}, where we used $H$ to denote a hyperplane in $\bbH^n$. We hope that this leads to no confusion.

\smallskip\noindent
\textit{Step 2.} Let $R >0$ be arbitrary. We will now show that one can pass to a finite-sheeted cover $M_R \to M$ that keeps $\Sigma$ 
fixed but so that the normal injectivity radius radius of $\Sigma \subseteq M_R$ is at least $R$. 
To achieve this we will exploit the subgroup separability from Theorem \ref{subgroupsep}.

\smallskip\noindent
\textit{Proof of Step 2.} We first make the following observation: If the normal injectivity radius of $\Sigma \subseteq M$ is at less than $R$, then there exists $\gamma \in \Gamma$ such that
\begin{equation}\label{eq: deck group when small normal injecvivity radius}
    d_{\mathbb{H}^n}(\gamma \cdot \Tilde{x}_0,\Tilde{x}_0) \leq 2\big(R+\diam(\Sigma)\big)
    \quad \text{and} \quad
    \gamma \notin \Gamma_{\Sigma},
\end{equation}
where $\Tilde{x}_0 \in \tilde{\Sigma}$ is some basepoint. Indeed, if the normal injectivity radius of $\Sigma \subseteq M$ is less than $R$, 
then  there exists a geodesic $\sigma:[0,1] \to M$ of length at most $2R$ such that
\[
    \sigma(i) \in \Sigma 
    \quad \text{and} \quad
    \sigma^\prime(i) \perp T_{\sigma(i)}\Sigma \quad \text{for } i=0,1.
\]
Let $x_0 \in \Sigma$ be the projection of the chosen basepoint 
$\Tilde{x}_0 \in \tilde \Sigma$. Choose a distance minimizing geodesic $\tau_i$ in $\Sigma$ from $x_0$ to $\sigma(i)$. Then the concatenation $\tau_0 \cdot \sigma \cdot \tau_1^{-1}$ is a loop based at $x_0$ of length at most $2(R+\diam(\Sigma))$. Clearly, this loop is \textit{not} homotopic to a loop in $\Sigma$. This proves the existence of an element $\gamma \in \pi_1(M,x_0) \cong \Gamma$ 
satisfying (\ref{eq: deck group when small normal injecvivity radius}).

Note that, for $R >0$ fixed, there are only finitely many elements $\gamma \in \Gamma$ satisfying the conditions in (\ref{eq: deck group when small normal injecvivity radius}). Therefore, 
by Theorem \ref{subgroupsep}, there is a finite index subgroup $\Gamma^\prime$ of $\Gamma$ containing $\Gamma_\Sigma$
such that $\gamma \notin \Gamma^\prime$
for all $\gamma\in \Gamma$ satisfying (\ref{eq: deck group when small normal injecvivity radius}). 

Since $\Gamma^\prime$ is a subgroup of $\Gamma$ of finite index, we know that ${\rm Stab}_{\Gamma^\prime}(\tilde{H})$ is a subgroup of finite index in 
${\rm Stab}_\Gamma(\tilde{H})$ and hence it is a cocompact torsion free lattice. In particular, ${\rm Stab}_{\Gamma^\prime}(\tilde{H})$ is a geometrically finite subgroup of $\Gamma$. Thus from property (3) in Step 1 we obtain a finite index subgroup $Q \leqslant \Gamma$ such that there is a retraction ${\rm retr}^\prime:Q \to {\rm Stab}_{\Gamma^\prime}(\tilde{H})$. As ${\rm Stab}_{\Gamma^\prime}(\tilde{H})={\rm Stab}_{\Gamma^\prime \cap Q}(\tilde{H})$, we obtain a retraction ${\rm retr}^\prime:\Gamma^\prime \cap Q \to {\rm Stab}_{\Gamma^\prime \cap Q}(\tilde{H})$ by restriction. Moreover, $\Gamma_\Sigma \leqslant {\rm Stab}_{\Gamma^\prime}(\tilde{H})$ implies $\Gamma_{\Sigma} \leqslant \Gamma^\prime \cap Q$. The finite cover $M^\prime \to M$ associated to $\Gamma^\prime \cap Q$ has the desired properties. For ease of notation, we will in the following simply write $\Gamma^\prime$ for $\Gamma^\prime \cap Q$. This completes the proof of Step 2.
\hfill$\blacksquare$

The construction in Step 2 yields a finite cover $M^\prime$ of $M$ and 
connected totally geodesic submanifolds $\Sigma \subseteq 
H^\prime \subseteq M^\prime=\Gamma^\prime\backslash \mathbb{H}^n$ 
such that the normal injectivity radius of $\Sigma \subseteq M^\prime$ is larger than an arbitrary multiple of 
$\diam(\Sigma)$. If $\Sigma$ is null-homologous in $H^\prime$ (and hence in $M^\prime$) we are done. So we assume that $\Sigma$ is \textit{not} null-homologous in $H^\prime$.

\smallskip\noindent
\textit{Step 3.}
Finally we show that there is a two-sheeted cover $\hat{M} \to M^\prime$ 
such that the preimage $\hat{\Sigma} \subseteq \hat{M}$ of $\Sigma$ is null-homologous in $\hat{M}$ and 
the preimage $\hat H$ of $H^\prime$ in $\hat M$ is connected.  

\smallskip\noindent
\textit{Proof of Step 3.}
In accordance with Fine--Premoselli \cite[Definition 2.2]{FP20} we say that $\Sigma \subseteq H^\prime$ \textit{separates} 
$H^\prime$ if $H^\prime \setminus \Sigma$ is disconnected. It is well-known that this can be detected algebraically. Namely, $\Sigma$ determines a class $[\Sigma] \in H_{n-2}(H^\prime;\bbZ/2\bbZ)$ and hence, by Poincaré duality, also a homomorphism 
\[
    \rho:\pi_1(H^\prime) \to \bbZ/2\bbZ
\]
that counts the number of intersections mod $2$ of a generic loop with $\Sigma$. Then $\Sigma$ separates $H^\prime$ 
if and only if $\rho$ is trivial (see for example \cite[proof of Lemma 2.3]{FP20}). Therefore, 
if $\rho$ is non-trivial, then the two-sheeted cover 
$\hat{H} \to H^\prime$ associated to $\ker(\rho)$ is connected, 
and the preimage $\hat{\Sigma} \subseteq \hat{H}$ of $\Sigma$ will separate  
$\hat{H}$. In particular, $\hat{\Sigma}$ is null-homologous in $\hat{H}$. Note that 
as $\pi_1(\Sigma)$ is contained in the kernel of $\rho$, the manifold $\hat \Sigma$ has precisely two connected 
components, each of which is isometric to $\Sigma$. 

Precomposing with the retraction ${\rm retr}^\prime:\Gamma^\prime\to {\rm Stab}_{\Gamma^\prime}(\tilde{H})$ we can extend $\rho$ to a homomorphism defined on $\Gamma^{\prime}$. Let 
$\hat\Gamma\leqslant \Gamma^\prime$ be the kernel of this homomorphism.
Then $\hat M=\hat\Gamma\backslash \mathbb{H}^n$ contains 
the two-sheeted cover $\hat{H} \to H^\prime$ of $H$ as an embedded totally geodesic submanifold. Since $\hat{\Sigma}$ is null-homologous in $\hat{H}$, it is also null-homologous in $\hat{M}$. Furthermore, as $\hat \Gamma\leqslant \Gamma^\prime$, the 
normal injectivity radius of $\hat \Sigma$ is at least $R$. 
This completes the proof of Step 3. \hfill$\blacksquare$

Recall that if $\Sigma$ is not homologous to zero then 
$\hat{\Sigma}$ has precisely two connected components, isometric to $\Sigma$, 
and the normal injectivity radius can \textit{not} become smaller. 
Therefore, this completes the proof of \Cref{prop: good submanifolds}.
\end{proof}

\section{Construction of the Einstein metric}\label{sec: Einstein metrics}

Let $M_k$ be a sequence of closed hyperbolic $n$-manifolds with a totally geodesic codimension two submanifold $\Sigma_k$ as in \Cref{prop: good submanifolds}, $X_k$ 
the cyclic $d$-fold covering of $M_k$ branched along $\Sigma_k$, and $\bar{g}_k$ the approximate Einstein metric on $X_k$ given by \Cref{prop: Properties of the approximate Einstein metric}.

The goal of this section is to prove \Cref{main thm}, that is, to show that $X_k$ admits a
negatively curved Einstein metric. By \Cref{Zeros of Phi are Einstein} it suffices to show that the Einstein operator  $\Phi_k=\Phi_{\bar{g}_k}$ defined in (\ref{eq: Def of Phi}) has a zero sufficiently close to the zero section. 
We will achieve this by an application of the Inverse Function Theorem. 

Recall from (\ref{eq: linearisation of Einstein}) that the linearization of the Einstein operator at the background metric $\bar{g}_k$ is given by
\[
    \mathcal{L}=(D\Phi_k)_{\bar{g}_k}=\frac{1}{2}\Delta_L+(n-1)\id.
\]

We will first show in \Cref{subsec: Linearized Einstein is invertible} that $\mathcal{L}$ is an invertible linear operator between suitable Banach spaces. \Cref{subsec: proof of main thm} then contains the proof of \Cref{main thm}.

\subsection{Invertibility of the linearized Einstein operator}\label{subsec: Linearized Einstein is invertible}

It is a classic result of Koiso \cite[Section 3]{Koi78} (also see \cite[Lemma 12.71]{Bes08}) that for a closed Einstein manifold with negative sectional curvature, the operator $\mathcal{L}$ has a uniform $L^2$-spectral gap (only depending on the negative upper curvature bound). By an adaptation of Koiso's argument, the same is also true for the approximate Einstein metrics $\bar{g}_k$.

\begin{lem}[$L^2$-spectral gap]\label{lem: L2 spectral gap}
There exists a constant $\lambda=\lambda(n,d) > 0$ such that for all sufficiently large $k$ we have
\[
    \lambda \int_{X_k}|h|^2 \, d\vol_{\bar{g}_k} \leq \int_{X_k} \langle \mathcal{L}h,h\rangle d\vol_{\bar{g}_k} 
\]
for all $h \in C^2\big(\Sym^2(T^\ast X_k)\big)$.
\end{lem}

For a detailed proof we refer the reader to \cite[Proposition 4.3]{FP20}, which is a bit more general than what we need here.

Fix a H\"older parameter $\alpha \in (0,1)$. We equip $C^{0,\alpha}\big(\Sym^2(T^\ast X_k) \big)$ with the hybrid norm
\begin{equation}\label{eq: hybrid 0-norm}
    ||f||_0:=\max\Big\{||f||_{C^{0,\alpha}(X_k,\bar{g}_k)},||f||_{L^2(X_k,\bar{g}_k)}\Big\}.
\end{equation}
Similarly, we equip $C^{2,\alpha}\big(\Sym^2(T^\ast X_k) \big)$ with the hybrid norm
\begin{equation}\label{eq: hybrid 2-norm}
    ||h||_2:=\max\Big\{||h||_{C^{2,\alpha}(X_k,\bar{g}_k)},||h||_{H^2(X_k,\bar{g}_k)}\Big\},
\end{equation}
where $||\cdot||_{H^2(X_k,\bar{g}_k)}$ is the Sobolev norm
\[
    ||h||_{H^2(X_k,\bar{g}_k)}:=\left(\int_{X_k}|h|^2+|\nabla h|^2+|\Delta h|^2 \, d\vol_{\bar{g}_k}\right)^{\frac{1}{2}}.
\]
Here the H\"older norm of a tensor is defined by the H\"older norm of the coefficients of the tensor in a harmonic chart defined on balls of a priori size (for a detailed account we refer to \cite[Proof of Proposition 2.5]{HJ22}).

Using the $C^0$-estimate from \Cref{lem: Nash-Moser} with the $L^2$-estimate from \Cref{lem: L2 spectral gap}, it is now straightforward to show that $\mathcal{L}$ is invertible (with universal constants).

\begin{prop}[$\mathcal{L}$ is uniformly invertible]\label{prop: L invertible}
There exists a constant $C=C(\alpha,n,d)$ with the following property. For all $k$ sufficiently large, the linearized Einstein operator
\[
	\mathcal{L}: \Big( C^{2,\alpha}\big(\Sym^2(T^*X_k)\big), ||\cdot||_2 \Big) \longrightarrow \Big( C^{0,\alpha}\big( \Sym^2(T^*X_k)\big), ||\cdot||_0 \Big)
\]
is invertible, and 
\[
	||\mathcal{L}||_{\rm op}, ||\mathcal{L}^{-1}||_{\rm op} \leq C,
\]
where \(||\cdot||_0\) resp. \(||\cdot||_2\) is the norm defined in (\ref{eq: hybrid 0-norm}) resp. (\ref{eq: hybrid 2-norm}).
\end{prop}

\begin{proof}
    It is clear that $||\mathcal{L}||_{\rm op}$ is bounded by a universal constant. It will suffice to prove the a priori estimate
    \(
        ||h||_2 \leq C||\mathcal{L}h||_0 
    \)
    for all $h\in C^{2,\alpha}\big(\Sym^2(T^*X_k)\big)$. Indeed, given the a priori estimate, standard arguments show that $\mathcal{L}$ is surjective; consequently $\mathcal{L}$ is invertible and $||\mathcal{L}^{-1}||_{\rm op} \leq C$ due to the a priori estimate.

    Clearly, $||h||_{L^2(X_k)} \leq C||\mathcal{L}h||_{L^2(X_k)}$ because $\mathcal{L}$ has a uniform $L^2$-spectral gap (\Cref{lem: L2 spectral gap}). Since $\mathcal{L}=\frac{1}{2}\Delta_L+(n-1)\id$, this $L^2$-estimate implies
    \begin{equation}\label{eq: H^2 bound}
        ||h||_{H^2(X_k)} \leq C||\mathcal{L}h||_{L^2(X_k)}.
    \end{equation}
    Moreover, the well-known Schauder estimates (see \cite[Proposition 2.5]{HJ22}) state
    \begin{equation}\label{eq: Schauder}
        ||h||_{C^{2,\alpha}(X_k)} \leq C\Big(||\mathcal{L}h||_{C^{0,\alpha}(X_k)}+||h||_{C^0(X_k)} \Big).
    \end{equation}
    The $C^0$-estimate (\ref{C^0}) together with (\ref{eq: H^2 bound}) yields
    \begin{equation}\label{eq: C^0 by hybrid norm}
        ||h||_{C^0(X_k)} \leq C\Big(||h||_{L^2(X_k)}+||\mathcal{L}h||_{C^0(X_k)} \Big)
        \leq  C\Big(||\mathcal{L}h||_{L^2(X_k)}+||\mathcal{L}h||_{C^0(X_k)}\Big).
    \end{equation}
    Keeping in mind the definitions (\ref{eq: hybrid 0-norm}) and (\ref{eq: hybrid 2-norm}) of the norms $||\cdot||_0$ and $||\cdot||_2$, the desired a priori estimate $||h||_2 \leq C||\mathcal{L}h||_0$ follows by combining (\ref{eq: H^2 bound}), (\ref{eq: Schauder}) and (\ref{eq: C^0 by hybrid norm}).
\end{proof}

\subsection{Proof of the Main Theorem}\label{subsec: proof of main thm}

The goal of this subsection is to present the proof of our main result.

\begin{thm}[Existence of Einstein metrics]\label{thm: existence of einstein metrics}
    For all sufficiently large $k$ there exists a metric $\hat{g}_k$ on $X_k$ such that
    \[
        \Ric(\hat{g}_k)+(n-1)\hat{g}_k=0 
        \quad \text{and} \quad
        \sec(X_k,\hat{g}_k) \leq -c(n,d) < 0.
    \]
    Moreover,
    \[
        ||\hat{g}_k-\bar{g}_k||_{C^{2,\alpha}(X_k,\bar{g}_k)} \xrightarrow{k \to \infty} 0.
    \]
\end{thm}

Before we come to the proof we point out that, for $k$ sufficiently large, 
the Einstein metric $\hat{g}_k$ can \textit{not} be locally symmetric. 
Indeed, $\sec(\Sigma_k,\bar{g}_k)=-u_{a(d)}^{-2} < -1$ by (\ref{eq: Ansatz for Einstein metric}) and \Cref{lem: cone singularities of Ansatz}. Thus $\sec(\Sigma_k,\hat{g}_k) < -1$ for all $k$ sufficiently large, and so $\hat{g}_k$ can \textit{not} be (real) hyperbolic. Moreover, by construction, the metric $\bar{g}_k$ is hyperbolic outside of a tubular neighborhood of $\Sigma_k$. Hence, outside of a tubular neighborhood of $\Sigma_k$, $\sec(X_k,\hat{g}_k)$ is very close to $-1$, and so  $\hat{g}_k$ can \textit{not} be complex- or quaternionic hyperbolic nor the Cayley plane.

In fact, in \Cref{sec: not locally symmetric} we will show that for a 
slightly restricted choice of the hyperbolic manifolds $M_k$, at most one of the cyclic 
branched coverings $X_k$ can \textit{not} admit any locally symmetric metric.

\begin{proof}
    We equip $C^{k,\alpha}\big(\Sym^2(T^\ast X_k)\big)$ with the norm $||\cdot||_k$ defined in (\ref{eq: hybrid 0-norm}) and (\ref{eq: hybrid 2-norm}) ($k=0,2$); $B(h,r)$ shall denote the balls with respect to these norms. 
    
    Any element in $B(\bar{g}_k,1/2) \subseteq C^{2,\alpha}\big(\Sym^2(T^\ast X_k)\big)$ is a positive definite $(0,2)$-tensor, that is, a Riemannian metric on $X_k$. Let $\Phi_k=\Phi_{\bar{g}_k}$ be the Einstein operator defined in (\ref{eq: Def of Phi}), which we consider as an operator
    \[
        \Phi_k: B(\bar{g}_k,1/2) \subseteq C^{2,\alpha}\big(\Sym^2(T^\ast X_k)\big) \to
        C^{0,\alpha}\big(\Sym^2(T^\ast X_k)\big).
    \]
    Denote by $\mathcal{L}=(D\Phi_k)_{\bar{g}_k}$ the linearization of $\Phi_k$ at the background metric $\bar{g}_k$. By \Cref{prop: L invertible} there exists a universal constant $C_0=C_0(\alpha,n,d)$ such that, for all $k$ sufficiently large, $\mathcal{L}$ is invertible with $||\mathcal{L}||_{\rm op}, ||\mathcal{L}^{-1}||_{\rm op} \leq C_0$. Moreover, by possibly enlarging $C_0$, it is clear that the map $g \mapsto (D\Phi_k)_g$ is $C_0$-Lipschitz. Therefore, applying (a quantitative version of) the Inverse Function Theorem implies that there exist constants $\varepsilon_0=\varepsilon_0(\alpha,n,d) > 0$ and $C=C(\alpha,n,d)$ with the following property: For each $f \in C^{0,\alpha}\big(\Sym^2(T^\ast X_k)\big)$ with $||f-\Phi_k(\bar{g}_k)||_0 \leq \varepsilon_0$ there exists a metric $g_f \in C^{2,\alpha}\big(\Sym^2(T^\ast X_k)\big)$ such that 
    \[
        \Phi_k(g_f)=f 
        \quad \text{and} \quad
        ||g_f-\bar{g}_k||_2 \leq C||f-\Phi_k(\bar{g}_k)||_0.
    \]
    Note that $\Phi_k(\bar{g}_k)=\Ric(\bar{g}_k)+(n-1)\bar{g}_k$. Hence it follows from \Cref{prop: Properties of the approximate Einstein metric}(i) and \Cref{cor: small L2 norm} that $||\Phi_k(\bar{g}_k)||_0 \to 0$ as $k \to \infty$. In particular, for all $k$ sufficiently large, $f=0$ satisfies $||f-\Phi_k(\bar{g}_k)|| \leq \varepsilon_0$. Therefore, there exists a metric $\hat{g}_k$ on $X_k$ such that
    \[
        \Phi_k(\hat{g}_k)=0
        \quad \text{and} \quad
        ||\hat{g}_k-\bar{g}_k||_2 \xrightarrow{k \to \infty}0.
    \]
    In particular, as $\sec(X_k,\bar{g}_k) \leq -c(n,d) < 0$ by \Cref{prop: Properties of the approximate Einstein metric}(iii), also $\sec(X_k,\hat{g}_k) < 0$ for all $k$ sufficiently large. Therefore, $\Phi_k(\hat{g}_k)=0$ implies $\Ric(\hat{g}_k)+(n-1)\hat{g}_k=0$ due to \Cref{Zeros of Phi are Einstein}. This completes the proof.
\end{proof}

For the formulation of the next remark, note that there is a natural action of the cyclic group $C_d$ of order $d$ on the $d$-fold branched cover $X_k$.

\begin{rem}\label{rem: conical Einstein metrics on M}
For all $k$ sufficiently large, the Einstein metric $\hat{g}_k$ on $X_k$ given by \Cref{thm: existence of einstein metrics} is $C_d$-invariant. In particular, for all $k$ sufficiently large (depending on $d$), the hyperbolic manifolds $M_k$ admit negatively curved Einstein metrics with a conical singularity and cone angle $\frac{2\pi}{d}$ along the codimension two submanifold $\Sigma_k \subseteq M_k$.
\end{rem}

\begin{proof}
In the proof of \Cref{thm: existence of einstein metrics}, the Einstein metric $\hat{g}_k$ was the zero of the Einstein operator $\Phi_k$ obtained from an application of the Inverse Function Theorem. Since the Inverse Function Theorem can be proved using the Banach Fixed Point Theorem, $\hat{g}_k$ is of the form $\bar{g}_k+\hat{h}_k$, where $\hat{h}_k$ is a fixed point of the operator
\[
        \Psi_k: C^{2,\alpha}\big(\Sym^2(T^\ast X_k)\big) \to
        C^{2,\alpha}\big(\Sym^2(T^\ast X_k)\big), \, h \mapsto h-\mathcal{L}^{-1}\big(\Phi_k(\bar{g}_k+h)\big).
\]
Using the definition (\ref{eq: Def of Phi}) of the Einstein operator, one can easily check that if a Riemannian metric $g$ on $X_k$ is $\varphi$-invariant for some $\varphi \in {\rm Isom}(X_k,\bar{g}_k)$, then also $\Phi_k(g)$ is $\varphi$-invariant. As the fixed point $\hat{h}_k$ is given by the limit $\lim_{i \to \infty} \Psi_k^i(0)$, this shows that $\hat{h}_k$, and hence $\hat{g}_k$, is ${\rm Isom}(X_k,\bar{g}_k)$-invariant. However, it is apparent from the construction of $\bar{g}_k$ explained in \Cref{subsec: approximate Einstein} that $\bar{g}_k$ is $C_d$-invariant. Therefore, also $\hat{g}_k$ is $C_d$-invariant.
\end{proof}

\section{Einstein manifolds not homeomorphic to locally symmetric spaces}\label{sec: not locally symmetric}

By \Cref{thm: existence of einstein metrics} there exist negatively curved Einstein metrics on some branched covers $X$ of 
certain hyperbolic manifolds $M$. The construction is valid for all covering degrees smaller than a number 
depending on $M$. As $M$ varies, this maximal covering degree can be arbitrarily large. 
The goal of this section is to show that 
for any dimension $n\geq 4$, we find infinitely many such branched coverings which 
are not homeomorphic to a locally symmetric manifold. 

We start with the following basic observation. 

\begin{prop}\label{prop: no complex hyperbolic metric}
Let $M$ be a closed hyperbolic $n$-manifold and $\Sigma \subseteq M$ a closed null-homologous totally geodesic submanifold of codimension two. 
Then the cyclic $d$-fold covering $X$ of $M$ 
branched along $\Sigma$ is \emph{not} homeomorphic to any locally symmetric manifold, except possibly hyperbolic manifolds.
\end{prop}

\begin{proof}
%
Arguing by contradiction, we assume that $X$ is homeomorphic to a locally symmetric manifold $N$ that is \textit{not} hyperbolic. 
We first observe that this locally symmetric manifold has to be of real rank one. 
Namely, since $X$ is aspherical, a locally symmetric metric on 
a manifold homeomorphic to $X$ is of non-positive
curvature. By a theorem of Wolf (Theorem 4.2 of \cite{W62}), a cocompact lattice in a semisimple Lie group 
of real rank $r$ contains a subgroup isomorphic to $\mathbb{Z}^r$. 
However, as $X$ carries a negatively curved metric \cite{GT87}, 
by Preissmann's theorem \cite[Théorème 10]{Pr42} (also see \cite[Theorem 3.2 in Chapter 12]{dC92}) any abelian subgroup of $\pi_1(X)$ is infinite cyclic.

It remains to show that $X$ is \textit{not} homeomorphic to any complex-, quaternionic- or Cayley-hyperbolic manifold. 
If $X$ is homotopy equivalent to a complex hyperbolic manifold $N$,  then there is a degree $d\geq 2$
map $\Pi:N\to M$. 
Since $M$ has constant negative curvature, 
the map $\Pi$ is homotopic to a harmonic map. But by a theorem of Sampson \cite{Sa86}, any harmonic map from 
a compact K\"ahler manifold into a real hyperbolic manifold is trivial in homology of dimension larger than 
two, which contradicts the fact that the degree of the map $\Pi$ is positive (unless $n=2$ and $N$ is also real hyperbolic).

By a celebrated result of Novikov \cite[Theorem 1]{Nov65}, the rational Pontryagin classes are a homeomorphism invariant. In particular, the Pontryagin numbers are a homeomorphism invariant. By a result of Lafont--Roy \cite[Theorem B]{LR07} all Pontryagin numbers of $X$ vanish, while it is a well-known consequence of the Hirzebruch proportionality principle \cite[Satz 2 and Equation (2)]{Hir56} that closed quaternionic- or Cayley-hyperbolic manifolds have some non-zero Pontryagin numbers (see \cite[Corollary 3]{LR07}). Therefore, $X$ can also \textit{not} be homeomorphic to a quaternionic- or Cayley-hyperbolic manifold. 
\end{proof}

As a consequence of Proposition \ref{prop: no complex hyperbolic metric} and
the work of Besson, Courtois and Gallot \cite{BCG95} we obtain.

\begin{cor}\label{rigidity}
If ${\rm dim}(X)=4$ and $X$ admits an Einstein metric as constructed in Theorem \ref{thm: existence of einstein metrics}
then $X$ is not homeomorphic to a locally symmetric manifold.
\end{cor}
\begin{proof}
Using the notations from Proposition \ref{prop: no complex hyperbolic metric}, if $X$ is homeomorphic 
to a locally symmetric manifold $M$ then $M$ is real hyperbolic. As $X$ admits an Einstein metric $g$, it is a consequence
of \cite[Théorème 9.6]{BCG95} (also see \cite[Corollary 4.6]{And10}) 
that $X$ is diffeomorphic to $M$ and $g$ is of constant curvature. 
However, the curvature of the Einstein metric $g$ on $X$ is not constant, from which the corollary follows.
\end{proof}

In the remainder of this section, which is independent of the rest of the article, 
we show that for 
any $n\geq 4$ there are infinitely many arithmetic hyperbolic manifolds $M$ of dimension $n$ 
which admit branched covers to which our 
construction of Einstein metrics applies, but such that 
at most one of these branched covers can be homeomorphic to a hyperbolic manifold. 
Together with \Cref{prop: no hyperbolic metric}, this completes the proof of Theorem \ref{main thm}.

We begin with collecting some more specific information on the 
standard arithmetic hyperbolic manifolds
used in our construction. 
Let $k$ be a totally real number field of degree $d$ over $\mathbb{Q}$ equipped with a fixed embedding
into $\mathbb{R}$ which we refer to as the identity embedding.
Let $V$ be an $(n+1)$-dimensional vector space over $k$ equipped with a quadratic form 
$q$ (with associated symmetric matrix $Q$) defined over $k$ which has signature $(n,1)$ at the identity
embedding, and signature $(n+1,0)$ at the remaining embeddings. Such a quadratic form 
is called \emph{admissible}. We require in the sequel that $q$ is \emph{anisotropic} over 
$\mathbb{Q}$. This means that $q=0$ has no rational solution.

Let $R_k$ be the ring of integers of the number field $k$ and 
let ${\rm O}(q,R_k)$ be the group of automorphisms of the quadratic form $q$ which 
are defined over $R_k$, that is,
\[{\rm O}(q,R_k):=\{X\in {\rm GL}_{n+1}(R_k)\mid X^t Q X=Q\}.\]

A subgroup $\Gamma$ of the isometry group ${\rm O}^+(n,1)$ of the hyperbolic space $\mathbb{H}^n$ is 
called an \emph{arithmetic group of simplest type}
if $\Gamma$ is commensurable with a conjugate of an arithmetic group 
${\rm O}(q,R_k)$. As the quadratic form $q$ is admissible
and anisotropic over $\mathbb{Q}$, an arithmetic group of simplest type $\Gamma$ is a cocompact
lattice in ${\rm O}^+(n,1)$. Thus 
$\Gamma \backslash \bbH^n$ is a compact hyperbolic orbifold with singularities corresponding to the fixed points of $\Gamma$. 
We refer to \cite[Example 6.30]{Em23} for more information.

\begin{example}\label{standardarith}
The quadratic form
\[
	q(x)=-\sqrt{2}x_0^2+x_1^2+\dots + x_n^2
\]
on $\bbR^{n+1}$ is defined over the quadratic extension $\mathbb{Q}(\sqrt{2})$ of $\mathbb{Q}$. 
Evaluation on the non-identity embedding $\mathbb{Q}(\sqrt{2})\to \mathbb{R}$ given by
$\sqrt{2}\to -\sqrt{2}$ shows that $q$ is admissible, moreover it is anisotropic over $\mathbb{Q}$. 
The upper paraboloid $\big\{x \in \bbR^{n+1} \, | \, q(x)=-1 \text{ and }x_0 > 0\big\}$ is a model for $\bbH^n$. 

The ring of integers of the number field $\bbQ(\sqrt{2})$ is the ring $\bbZ[\sqrt{2}]$ and hence 
\[
	{\rm O}(q,\mathbb{Z}[\sqrt{2}])={\rm O}(q) \cap \GL_{n+1}\big(\bbZ\big[\sqrt{2}\big]\big)
\]
is a cocompact lattice in ${\rm O}^+(n,1)$. 
\end{example}

Standard theory of quadratic forms (see \cite{La73}) provides an equivalence over $k$ 
of the quadratic form $q$ to an admissible diagonal quadratic form. Thus we may assume without
loss of generality that 
\[q(x)=-a_0 x_0^2+a_1 x_1^2+ \cdots + a_nx_n^2\]
with $a_i\in k, a_i>0$. Put $\Gamma={\rm O}(q,R_k)$. 

Let $\iota \in {\rm Isom}(\bbH^n)$ be  the geometric involution  
that acts via reflection in the $x_1$-variable, that is,
\[
	\iota(x_0,x_1,x_2,\dots,x_n)=(x_0,-x_1,x_2,\dots,x_n).
\]
Then $\tilde{H}:={\rm Fix}(\iota)=\{x \in \bbH^n \, | \, x_1=0\}$ is a hyperplane. 
The quadratic form 
\[q_0(x)=-a_0x_0^2+a_2x_2^2+ \cdots +a_nx_n^2\]
on the linear subspace $V_0=\{x_1=0\}$ of $V$ defined over $k$ 
is admissible and anisotropic over $\mathbb{Q}$. 
Under the obvious identifications we then have 
${\rm Stab}_{\Gamma}(\tilde{H})={\rm O}^+(n-1,1) \cap {\rm O}(q,R_k)$, so that, by the same reason as above, 
the quotient ${\rm Stab}_{\Gamma}(\tilde{H}) \backslash \tilde{H}$ is compact. This means that, in the terminology of \Cref{subsec: subgroup separability}, $\tilde{H}$ is a $\Gamma$-hyperplane. Furthermore, we have $\iota \Gamma \iota =\Gamma$. 

Consider the sequence 
$\Gamma_m \vartriangleleft \Gamma$ of congruence subgroups defined as the kernel of the 
natural homomorphism
\[
    \Gamma\to \GL_{n+1}\big(R_k\big)\to 
\GL_{n+1}\big(R_k/{\mathcal O}_m\big)
\]
where ${\mathcal O}_m$ is a sequence of mutually distinct prime ideals in $R_k$. 
For sufficiently large $m$ the group $\Gamma_m$ is sufficiently deep and hence torsion free.
The quotient manifold 
$ N_m= \Gamma_m\backslash \mathbb{H}^n$ is a standard arithmetic hyperbolic 
manifold. Moreover, by construction, $N_m$ is oriented. 

As kernels are normal subgroups, one easily checks $\iota \Gamma_m\iota^{-1}=\Gamma_m$.  
It follows that $\iota$ descends to an isometric 
involution of $N_m=\Gamma_m\backslash \mathbb{H}^n$, again denoted by
$\iota$. The fixed point set of this involution 
is a (possibly disconnected) totally geodesic submanifold of codimension one. Fix a component
$H$ of this submanifold. We may assume that $H$
is the projection to $N_m$ of the $\Gamma_m$-hyperplane $\tilde H$. Following
the construction in Section \ref{sec: good submanifolds}, we know that 
${\rm Stab}_{\Gamma_m}(\tilde H)$ is a virtual retract of 
$\Gamma_m$. Let $\Gamma_m^\prime \leqslant \Gamma_m$ be a finite index subgroup containing 
the fundamental group ${\rm Stab}_{\Gamma_m}(\tilde H)$ of $H$ 
which retracts onto ${\rm Stab}_{\Gamma_m}(\tilde H)$.

\begin{lem}\label{fi}
There exists a $\iota$-invariant finite index subgroup $\Gamma_m^0 \leqslant \Gamma_m^\prime$ 
which contains 
${\rm Stab}_{\Gamma_m}(\tilde H)$.
\end{lem}

In particular, by restricting the retraction ${\rm ret}:\Gamma_m^\prime\to {\rm Stab}_{\Gamma_m}(\tilde H)$ to $\Gamma_m^0$, we see that $\Gamma_m^0$ also retracts onto ${\rm Stab}_{\Gamma_m^0}(\tilde H)={\rm Stab}_{\Gamma_m}(\tilde H)$.

\begin{proof}
As $\Gamma_m^\prime \leqslant \Gamma_m$ has finite index and $\Gamma_m$ is $\iota$-invariant, 
$\Gamma_m^0:=\Gamma_m^\prime \cap \iota \Gamma_m^\prime \iota^{-1}$ is a $\iota$-invariant finite index subgroup of 
$\Gamma_m$. Moreover, by inspecting the action of the differential, one can check that $\iota^{-1} {\rm Stab}_{\Gamma_m}(\tilde H) \iota \subseteq {\rm Stab}_{\Gamma_m}(\tilde H) \subseteq \Gamma_m^\prime$. This then implies ${\rm Stab}_{\Gamma_m}(\tilde H) \leqslant \Gamma_m^0$.
\end{proof}

The group $\Gamma_m^0$ is invariant under conjugation by $\iota$, and this action 
of $\iota$ on $\Gamma_m^0$ is nontrivial. Thus  
$\iota$ acts as an isometric involution on $M_m=\Gamma_m^0\backslash \mathbb{H}^n$. Its fixed point set  
is a disjoint union of totally geodesic embedded hyperplanes containing the quotient $H$ 
of $\tilde H$ under the action of ${\rm Stab}_{\Gamma_m}(\tilde H)$. 

By the construction in Section \ref{sec: Einstein metrics}, by perhaps passing to a 
two-sheeted covering $\hat M_m$ of $M_m$, we
may assume that the preimage $\hat H$ of $H$ in $\hat M_m$ 
contains a totally geodesic embedded hyperplane $\hat \Sigma_m$ which is homologous to zero
and consists of at most two connected components. 
The involution $\iota$ may not lift to $\hat M_m$, but it lifts to 
the covering $\tilde M_m$ of $\hat M_m$ of degree at most two with fundamental group 
$\pi_1(\hat M_m)\cap \iota \pi_1(\hat M_m)\iota^{-1}$.
Note that as the hyperplane $H$ in $M_m$ is contained in the fixed point set of the involution
$\iota$, if $\pi_1(\hat M_m)<\pi_1(M_m)$ is not invariant under conjugation by $\iota$, then 
the preimage of $\hat H$ in $\tilde M_m$ consists of 
two components of $\hat H$, each of which contains a totally geodesic null homologous
hyperplane as required in the construction in the beginning of this article. 
 
Using this construction, \Cref{main thm} is an immediate consequence of \Cref{thm: existence of einstein metrics},
\Cref{prop: no complex hyperbolic metric} and the following main result of this section.

\begin{thm}\label{prop: no hyperbolic metric}
Let $M$ be an oriented closed hyperbolic manifold of dimension $n\geq 4$ 
and let $H\subset M$ to a totally geodesic embedded hyperplane. Assume that
$H$ is contained in the fixed point set of an orientation reversing isometric
involution $\iota$ and that 
$H$ contains a (possibly disconnected) embedded totally geodesic 
hyperplane $\Sigma$ which is homologous to zero in $H$. 
Then for at most 
one $d\in 4\mathbb{Z}$, the cyclic $d$-fold covering of $M$ branched along $\Sigma$ can be
homeomorphic to a hyperbolic manifold. 
\end{thm}

\begin{rem}
Building on the results in this section, in forthcoming work, we show that for $n\geq 4$, 
no nontrivial branched cover of a closed hyperbolic $n$-manifold admits a hyperbolic metric. We refer
to \cite{KS12} for a closely related result. 
\end{rem}

The remainder of this article is devoted to the proof of \Cref{prop: no hyperbolic metric}. 
It is inspired by \cite[Remark 3.6]{GT87}, though it does not directly follow from it. 
The section is divided into two subsections. 
We always consider a degree $d$ branched covering $X$ of $M$ for an \emph{even} number 
$d\geq 2$, and we assume that $X$ admits 
a hyperbolic metric.

\subsection{Fixed point sets of isometries}\label{fixedpointset}  

Let $M$ be as in the statement of Theorem \ref{prop: no hyperbolic metric}, containing
the hypersurface $H\supset \Sigma$.  By assumption, $\Sigma$ bounds a submanifold $H_0\subset H$.
Put $H_1=H\setminus H_0$. 

The $d$-fold covering $X$ of $M$ branched along the totally geodesic submanifold 
$\Sigma \subset H\subset M$ can be realized as follows. 
Let $M_{\rm cut}$ be obtained from $M$ by cutting along $H_0$, that is, $M_{\rm cut}$ is the metric completion of
$M-H_0$. Thus $M_{\rm cut}$
is a compact (topological) 
manifold whose boundary consists of two copies $H_0^{-}$ and $H_0^{+}$ of $H_0$ intersecting in $\Sigma$. 
The manifold $X$ 
is obtained by gluing $d$ copies $M_{\rm cut}^{1},\dots,M_{\rm cut}^{d}$ of $M_{\rm cut}$ along the boundary, so that the copy of $H_0^+$ in $M_{\rm cut}^{i}$ is glued to the copy of $H_0^{-}$ in $M_{\rm cut}^{i+1}$ (where the superscripts $i$ are taken ${\rm mod}\, d$).

Let $\iota=\iota_M:M\to M$ be the isometric involution 
whose fixed point set contains 
$H \subseteq {\rm Fix}(\iota)$. 
Since locally near $H$, $\iota_M$ acts as a reflection in $H$, it exchanges the two components of $U\setminus H$ where
$U$ is a tubular neighborhood of $H$ in $M$. Thus $\iota_M$ acts as an involution on $M_{\rm cut}$ which 
exchanges $H_0^{+}$ and $H_0^{-}$ and fixes $W={\rm Fix}(\iota_M)\setminus H_0\supseteq H_1$.

As a consequence, $\iota_M$ induces an 
involution $\iota$ of $X$ with the property that $\iota(M_{\rm cut}^i)=M_{\rm cut}^{d+2-i}$ 
and so that the restrictions $\iota:M_{\rm cut}^{i} \to M_{\rm cut}^{d+2-i}$ are identified with 
$\iota:M_{\rm cut} \to M_{\rm cut}$ (superscripts are again taken ${\rm mod}\, d$).

Let $\zeta$ be a generator of the cyclic deck group of $X\to M$. 
It cyclically permutes the 
copies $M^1_{\rm cut},\dots,M_{\rm cut}^d$ of $M_{\rm cut}$ in $X$. 
Define
$j=\zeta \circ \iota$ (read from right to left).

\begin{fact}\label{factfixed}
The fixed point set of $j$ in $X$ is the union 
$H_0^{d,1}\cup H_0^{d,1+d/2}$ of the copies of $H_0\subset H$ 
in $M_{\rm cut}^1$ and $M_{\rm cut}^{1+d/2}$, and the copies $H_0^{d,1}$ and $H_0^{d,1+d/2}$ of $H_0$ are
glued along $\Sigma$
(see \Cref{fig: branched cover and involutions}). 
\end{fact}

\begin{figure}[ht]
	\begin{center}
		\includegraphics[width=12cm]{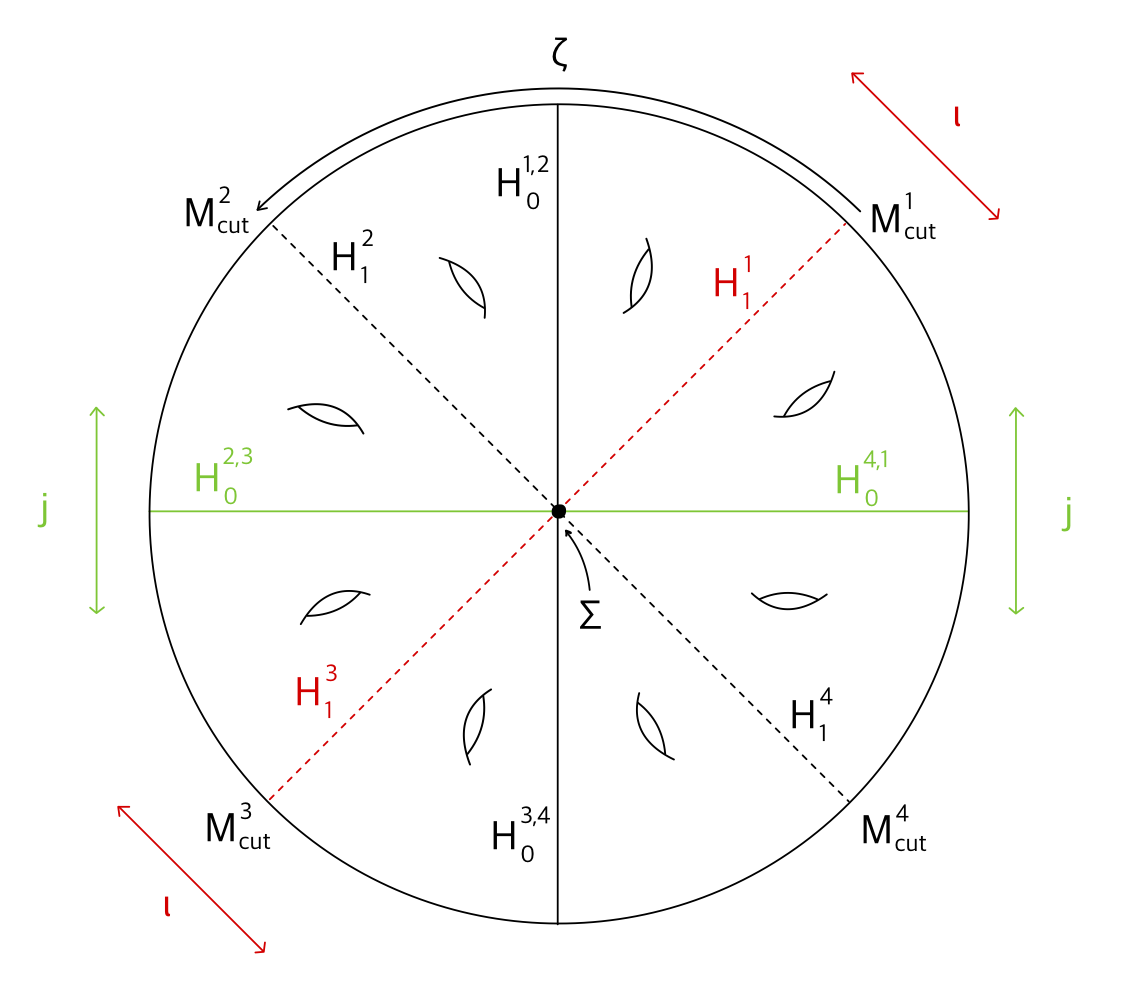}
		\caption{The cyclic $4$-fold branched cover. The involutions $\iota$ and $j$ act via reflection along the colored submanifolds and $\zeta$ via rotation around $\Sigma$.}
		\label{fig: branched cover and involutions}
	\end{center}
\end{figure}

The fixed point set of each of the involutions $\zeta^i \circ j \circ \zeta^{-i}$ ($i=0,\dots,d-1$) 
is the embedded submanifold $\zeta^i({\rm Fix}(j))$ of $X$. Their union cuts $X$ up into the $d$ copies of $M_{\rm cut}$.
We call any diffeomorphism of $X$ contained in the finite group of diffeomorphisms
of $X$ generated by $j$ and $\zeta$ an \emph{admissible} diffeomorphism of $X$.

By Mostow rigidity, any homotopy self-equivalence $\sigma$ of $X$ is homotopic to a unique isometry $\sigma^\#$ of $X$.
Furthermore, by uniqueness, the map
\[
	{\rm Homeo}(X) \to {\rm Isom}(X), \, \sigma \mapsto \sigma^\#
\]
which associates to a homeomorphism the unique isometry homotopic to it 
is a group homomorphism.
The following result relates the finite group of admissible diffeomorphisms of $X$ to the 
corresponding finite group of isometries for the hyperbolic metric. It 
seems be known to the experts, and it was claimed in \cite{GT87} for the generator
$\zeta$ of the deck group of $X\to M$ (except for Remark 3.4, this is not used in \cite{GT87}).  
In view of the fact that in the presence of fixed point sets of 
positive dimension, a finite group of diffeomorphisms of a hyperbolic manifold of dimension $n\geq 3$
need not be conjugate to its isometric realization (see the main result of \cite{CLW18} for the case 
of finite groups of homeomorphisms and the included remark about the case of diffeomorphisms), 
we present a proof.

\begin{prop}\label{fixed}
Let $\phi$ be 
an admissible diffeomorphism of $X$ 
and let $\phi^\#$ be the 
isometry of $X$ homotopic to $\phi$. Then 
the fixed point set 
${\rm Fix}(\phi^\#)\subset X$ of 
$\phi^\#$ 
is (abstractly) diffeomorphic to ${\rm Fix}(\phi)$. 
Moreover, ${\rm Fix}(\phi^\#)$ and ${\rm Fix}(\phi)$ are freely homotopic inside $X$.
\end{prop}

\begin{proof}
Let $\phi$ be any nontrivial admissible diffeomorphism of $X$. 
Then either $\phi$ is an orientation reversing involution whose fixed point set is 
a disjoint union of submanifolds of codimension one contained 
in the preimage of ${\rm Fix}(\iota_M)$, 
or it is a power of $\zeta$ with fixed point set $\Sigma$. 
In particular, the fixed point set of $\phi$ is a (possibly disconnected) orientable hyperbolic manifold of dimension
$n-2$ or $n-1$ containing $\Sigma$. 

We first observe that $\phi^\#$ does have fixed points. 
Indeed, otherwise $X \to \langle \phi^\# \rangle \backslash X $ 
would be a finite-sheeted covering map between manifolds. 
As $X$ is a closed hyperbolic manifold, it is aspherical and $\pi_1(X)$ has trivial center. 
But then \cite[Lemma 2.2]{HJ24} implies that $\phi^\#$ can \textit{not} be homotopic to a map of finite order and non-empty fixed point set, contradicting the fact that $\phi^\#$ is homotopic to $\phi$.

Since a component $Z^\#$ 
of ${\rm Fix}(\phi^\#)$ is a totally geodesic 
submanifold of $X$ containing the unique closed geodesic in $X$ freely homotopic 
to an element of $\pi_1(Z^\#)$, no two distinct components of 
${\rm Fix}(\phi^\#)$ can be freely homotopic. As a consequence, it suffices to show 
that for every component $Z$ of ${\rm Fix}(\phi)$ (or $Z^\#$ of ${\rm Fix}(\phi^\#)$)
there exists a component $Z^\#$ of ${\rm Fix}(\phi^\#)$ 
(or $Z$ of 
${\rm Fix}(\phi)$)
which is 
diffeomorphic and freely homotopic to $Z$ (or $Z^\#$). 

Thus let $Z$ be a component of ${\rm Fix}(\phi)$ and 
choose a basepoint $x\in Z$. 
Let $\phi_*$ be the automorphism of $\pi_1(X,x)$ induced by $\phi$. 
We divide the proof into six  steps. 

\smallskip\noindent
{\bf Claim 1.} \emph{We have ${\rm Fix}(\phi_\ast)=\pi_1(Z,x)$.}

\smallskip\noindent
\emph{Proof of Claim 1.}  
The hyperbolic metric on $M$ lifts to a hyperbolic cone metric $h$ on 
$X$ which is smooth away from $\Sigma$ and with cone angle $2 d\pi$ along $\Sigma$. 
Thus $h$ is locally a ${\rm CAT}(-1)$-metric. 
Therefore every homotopy class
$\alpha\in \pi_1(X,x)$ has a unique geodesic representative for the metric $h$ which 
is a geodesic loop based at $x$.
The map $\phi$ is an isometry for $h$ fixing 
$Z$ pointwise.

Note that
the inclusion $\pi_1(Z,x) \leqslant {\rm Fix}(\phi_\ast)$ trivially holds. 
To prove equality, we argue by contradiction and assume that there exists a class 
$[\gamma] \in {\rm Fix}(\phi_\ast) \setminus \pi_1(Z,x)$. 
This class is represented by a unique geodesic loop $\gamma$ based at $x$ that does 
\textit{not} entirely lie in $Z$. As $\phi$ is an isometry, $\phi(\gamma)$ 
is the geodesic representative of $\phi_\ast([\gamma])=[\gamma]$, 
and hence $\phi(\gamma)=\gamma$ by uniqueness. As $Z$ is a connected component of 
${\rm Fix}(\phi)$, we get $\gamma \subseteq Z$, contradicting 
$[\gamma] \notin \pi_1(Z,x)$. 
\hfill$\blacksquare$

\smallskip

The argument in the proof of Claim 1 also applies to the 
map $\phi^\#$ as an isometry for the hyperbolic metric on $X$ and shows that if 
$Z^\#$ is any component of the fixed point set ${\rm Fix}(\phi^\#)$ of $\phi^\#$,
which is a totally geodesic submanifold of $X$, and if
$y\in Z^\#$, then $\pi_1(Z^\#,y)={\rm Fix}(\phi^\#_\ast)\subset \pi_1(X,y)$.  

Since $\phi^\#$ has fixed points, by changing the hyperbolic metric with an isotopy, 
that is, replacing it by its pullback by a diffeomorphism of $X$ isotopic to the identity,
we may assume that $x\in {\rm Fix}(\phi^\#)$. Then $\phi^\#$ induces an automorphism
$\phi_*^\#$ of $\pi_1(X,x)$. 

\smallskip\noindent
{\bf Claim 2:} \emph{Let $[\gamma]\in \pi_1(Z,x)$ and let $\gamma^\#$ be the closed geodesic 
for the hyperbolic metric which is freely homotopic to $\gamma$; then 
$\gamma^\#\subset {\rm Fix}(\phi^\#)$.}

\smallskip\noindent
\emph{Proof of Claim 2.} 
Since $\phi$ and $\phi^\#$ are homotopic, there exists an element 
$\alpha\in \pi_1(X,x)$ such that $\phi_*^\#=\alpha \phi_*\alpha^{-1}$. 
As $[\gamma]\in \pi_1(Z,x)\subset {\rm Fix}(\phi_*)$ we know that
$\phi_*^\#([\gamma])$ is conjugate to $[\gamma]$. In other words, $\phi_*^\#$ preserves
the conjugacy class of $[\gamma]$.

Let $\gamma^\#$ be the unique oriented closed geodesic for the hyperbolic metric on $X$ 
in the free homotopy class of $[\gamma]$. 
Since $\phi^\#$ preserves the conjugacy class of $[\gamma]$ and is an isometry, it 
preserves $\gamma^\#$ as an oriented unparameterized circle. 

We argue by contradiction and we assume that $\gamma^\#\not\subset {\rm Fix}(\phi^\#)$.
If there are no fixed points of $\phi^\#$ on $\gamma^\#$ then as $\phi^\#$ preserves the 
hyperbolic norm of 
the tangent of $\gamma^\#$, it acts on the immersed  circle
$\gamma^\#\subset X$ as a nontrivial rotation. Then $\gamma^\#$ admits a lift $\tilde \gamma^\#$ 
to the universal covering $\mathbb{H}^n$ of $X$ so that a lift $\widetilde{\phi^\#}$ of $\phi^\#$ 
acts on $\tilde \gamma^\#$ as
a nontrivial translation. But any isometry of $\mathbb{H}^n$ which preserves a geodesic and 
acts on it as a non-trivial translation is loxodromic and hence fixed point free. 
This violates the fact that $\phi^\#$ and hence $\widetilde{\phi^\#}$ have fixed points.

As a conclusion, the restriction of $\phi^\#$ to $\gamma^\#$ has 
fixed points. Let $y\in \gamma^\#$ be such a fixed point. Since $\phi^\#$ preserves 
$\gamma^\#$ as a set, the differential 
$d_y\phi^\#$ of 
$\phi^\#$ at $y$ maps the (oriented) tangent $v$ of $\gamma^\#$ at $x$ to $\pm v$. If $d_y\phi^\#(v)=v$
then $\gamma^\#\subset {\rm Fix}(\phi^\#)$ 
since an isometry maps geodesics parameterized
by arc length to geodesics parameterized by arc length, and geodesics are determined by
their tangent at a single point. This contradicts the assumption $\gamma^\#\not\subset {\rm Fix}(\phi^\#)$.

Therefore 
$d_y\phi^\#(v)=-v$ and $\phi^\#$ reverses the orientation of $\gamma^\#$. 
Then 
$[\gamma^\#]$ is conjugate to its inverse in $\pi_1(X,x)$. 
This is equivalent to stating that there exists an element of $\pi_1(X,x)$ acting as the deck group 
of $X$ on $\mathbb{H}^n$ which exchanges the endpoints in the ideal boundary 
$\partial \mathbb{H}^n$ of $\mathbb{H}^n$ 
of a lift of $\gamma^\#$, 
contradicting the fact that any isometry with this property has a fixed point. Together this
completes the proof of Claim 2.
\hfill $\blacksquare$

\smallskip\noindent
{\bf Claim 3:} \emph{Up to changing the hyperbolic metric with an isotopy, we have
$\phi_*=\phi_*^\#$, in particular
${\rm Fix}(\phi_*^\#)={\rm Fix}(\phi_*)$.}

\smallskip\noindent
\emph{Proof of Claim 3.} 
Let $\gamma\subset Z$ be a (nontrivial) closed geodesic for the hyperbolic cone metric
$h$ on $X$. Note that such a geodesic exists since 
the dimension of each component of ${\rm Fix}(\phi)$ is 
at least two and $Z$ is totally geodesic for $h$. 
Let $x\in \gamma$ and let as before $\gamma^\#$ 
be the closed geodesic 
for the hyperbolic metric on $X$ which is freely homotopic to $\gamma$. 

Choose a point $x^\#\in \gamma^\#$ and an embedded arc 
$a:[0,1]\to X$, smooth up to and including the 
endpoints, which connects $x$ to $x^\#$ and such that $a \circ \gamma \circ a^{-1}$ 
(read from right to left) is homotopic 
to $\gamma^\#$ in $\pi_1(X,x^\#)$. Let $N$ be a tubular neighborhood
of $a$. There exists a smooth isotopy $[0,1]\times X\to X$ of $X$ which is the identity outside of 
$N$ and pushes the point $x$ along $a$. Let $\Lambda$ be the endpoint map of this isotopy. 
Then $\Lambda$ maps $\gamma$ to a based loop at $x^\#$ which is homotopic to 
$\gamma^\#$ and hence up to replacing the hyperbolic metric by its pull-back under $\Lambda$, 
we may assume that $x\in \gamma^\#$ and that the homotopy classes of 
$\gamma$ and $\gamma^\#$ in $\pi_1(X,x)$ coincide. 

Recall that $\phi_*^\#=\alpha \phi_*\alpha^{-1}$ for some $\alpha\in \pi_1(X,x)$
(the element $\alpha$ may have changed in the course of this proof). 
Since 
\[[\gamma]=[\gamma^\#]=\alpha [\gamma] \alpha^{-1}\]
we know that $\alpha$ centralizes the homotopy class $[\gamma]$ of 
$\gamma$. As $\pi_1(X,x)$ is the fundamental group of a hyperbolic manifold,
the centralizer of $[\gamma]$ equals the infinite cyclic group generated by 
$[\gamma]$. In particular, $\phi_*(\alpha)=\alpha$ since $[\gamma] \in {\rm Fix}(\phi_*)$, 
moreover $\phi^\#_*$ preserves the fixed point set $\pi_1(Z,x)$ of $\phi_*$. 

The component $Z$ of the fixed point set of $\phi$ is a closed oriented hyperbolic 
manifold of dimension at least two. Thus any nontrivial inner automorphism of 
$\pi_1(Z,x)$ has infinite order. Now $\phi^\#$ is an isometry of finite order 
and hence the order of $\phi_*^\#$ is finite as well. But 
$\phi_*^\#(\beta)=\alpha \beta \alpha^{-1}$ for all $\beta\in \pi_1(Z,x)$ and consequently
$\alpha=e$ and $\phi_*=\phi^\#_*$. 
This completes the proof of Claim 3.
\hfill$\blacksquare$

\smallskip

We showed so far that 
for every component 
$Z$ of ${\rm Fix}(\phi)$ there exists a component $Z^\#$ of ${\rm Fix}(\phi^\#)$
whose fundamental group 
is isomorphic to the fundamental group of $Z$.
Each component $Z$ of ${\rm Fix}(\phi)$ and corresponding component $Z^\#$ of ${\rm Fix}(\phi^\#)$
is naturally equipped with a hyperbolic metric. 
Its dimension equals the cohomological dimension of its fundamental group.
Thus if the dimension of $Z$ is at least three, then by Mostow rigidity, the manifolds 
$Z$ and $Z^\#$ are isometric and freely homotopic inside $X$.
If the dimension of $Z$ equals two then the manifolds $Z$ and $Z^\#$ are diffeomorphic
as the diffeomorphism type of a closed surface is determined by its fundamental group. 
Furthermore, $Z$ and $Z^\#$ are freely homotopic inside $X$. 

It remains to show that there is no component of ${\rm Fix}(\phi^\#)$ which is not freely 
homotopic to a component of ${\rm Fix}(\phi)$. This is carried out in the rest of this proof.

\smallskip\noindent
{\bf Claim 4:} \emph{$X\setminus {\rm Fix}(\phi)$ is aspherical.}

\smallskip\noindent
\emph{Proof of Claim 4.} 
As $X$ is a closed hyperbolic manifold by assumption,
its universal covering $\tilde X$ is diffeomorphic to $\mathbb{R}^n$. Furthermore, 
${\rm Fix}(\phi)\subset X$ is an embedded closed 
totally geodesic submanifold for the ${\rm CAT}(-1)$-hyperbolic cone metric $h$ on $X$ 
whose codimension  either equals 
one or two. The preimage $Y$ of $X\setminus {\rm Fix}(\phi)$ in $\tilde X$ 
is the complement in $\tilde X$ of a countable union of properly embedded submanifolds diffeomorphic
either to $\mathbb{R}^{n-1}$ or to $\mathbb{R}^{n-2}$.

If the codimension of these subspaces equals one then
$Y$ is a disjoint union of countably many contractible spaces. If the codimension of
these subspaces equals two then $Y$ is 
homotopy equivalent to the wedge of countably
many circles, each corresponding to a loop encircling one of the codimension two complementary subspaces.
Hence $Y$ is aspherical. Since $Y$ is a covering of $X\setminus {\rm Fix}(\phi)$, the space 
$X\setminus {\rm Fix}(\phi)$ is aspherical as well. 
%
\hfill $\blacksquare$

\smallskip

\smallskip\noindent
{\bf Claim 5:} \emph{$X\setminus {\rm Fix}(\phi)$ has the homotopy type of a finite
CW-complex, and its fundamental group is center free.}

\smallskip\noindent
\emph{Proof of Claim 5.} 
There are two cases possible for the map $\phi$. In the first case, ${\rm Fix}(\phi)$ is a 
finute disjoint union of compact codimension one submanifolds 
in $X$, and in the second case, we have ${\rm Fix}(\phi)=\Sigma$. 
In both cases, $X\setminus {\rm Fix}(\phi)$ is 
homotopy equivalent to a compact 
manifold with boundary, which can be chosen to be the complement of a small open tubular neighborhood of 
${\rm Fix}(\phi)$. Hence $X\setminus {\rm Fix}(\phi)$ 
has the homotopy type of a finite CW-complex.

To see  that $\pi_1(X\setminus {\rm Fix}(\phi))$ is center free, note that  
if ${\rm Fix}(\phi)$ is a disjoint union of hyperplanes, then
cutting 
$X$ open along the corresponding components of ${\rm Fix}(\phi^\#)$ yields a 
(possibly disconnected) compact
hyperbolic manifold $N$ with totally geodesic boundary whose fundamental group is a torsion free hyperbolic
group and hence center free. 

If ${\rm Fix}(\phi)=\Sigma$ then putting $G=\pi_1(X\setminus {\rm Fix}(\phi))$,  
the homomorphism
$\rho:G\to \pi_1(X)$ induced by the inclusion $X\setminus {\rm Fix}(\phi)\to X$ is surjective. Thus 
as $\pi_1(X)$ is torsion free and center free, 
an element in the center of 
$G$ is contained in the kernel of the homomorphism $\rho$ and hence it is contained in 
the center of ${\rm ker}(\rho)$. 
But the kernel of $\rho$ 
is the fundamental group of the preimage $Y$ 
of $X\setminus \Sigma$ in the universal covering $\mathbb{H}^n$ of $X$. As 
$Y$ has the homotopy type of a countable wedge of circles, this fundamental group 
is an infinitely generated free group and hence center free. 
\hfill$\blacksquare$

\smallskip

The following claim completes the proof of the proposition.

\smallskip\noindent
{\bf Claim 6:} \emph{There can not be any component of
${\rm Fix}(\phi^\#)$ that is not freely homotopic to a component of ${\rm Fix}(\phi)$. }

\smallskip\noindent
\emph{Proof of Claim 6.} 
We showed so far that there exists a union $Q$ of components of ${\rm Fix}(\phi^\#)$ which is 
abstractly diffeomorphic to ${\rm Fix}(\phi)$ and freely
homotopic to ${\rm Fix}(\phi)$ in $X$. The manifolds 
$X\setminus {\rm Fix}(\phi)$ and 
$X\setminus Q$ have isomorphic fundamental groups, and
by Claim 4 and its analog for $X\setminus Q$, they are 
aspherical. As a consequence, $X\setminus {\rm Fix}(\phi)$ and $X\setminus Q$
are homotopy equivalent. 

Let $\Lambda:X\setminus {\rm Fix}(\phi)\to X\setminus Q$ be a homotopy equivalence,
with homotopy inverse $\Lambda^{-1}$. 
The map $\phi$ acting on $X\setminus {\rm Fix}(\phi)$ is homotopic to the 
map $\hat \phi=\Lambda^{-1}\circ \phi^\#\circ \Lambda$, read from right to left.
Thus via an identification of $\pi_1(X\setminus {\rm Fix}(\phi))$ with 
$\pi_1(X\setminus Q)$ via the homotopy equivalence $\Lambda$, the maps $\phi$ and 
$\phi^\#$ induce the same outer automorphisms of $\pi_1(X\setminus {\rm Fix}(\phi))$. 
Furthermore, by construction, $\phi$ and $\phi^\#$ have the same order, say $m\geq 2$.

We now follow \cite[Lemma 2.2]{HJ24}.  
The finite order diffeomorphism $\phi$ restricts to a fixed point free finite order
diffeomorphism on $X\setminus {\rm Fix}(\phi)$. 
Let $\bar X=\langle \phi\rangle\backslash (X\setminus {\rm Fix}(\phi))$ 
be the quotient of $X$ under the free
action of $\phi$. There exists an exact sequence
\[1\to \pi_1(X\setminus {\rm Fix}(\phi)) \to \pi_1(\bar X)\to \mathbb{Z}/m\mathbb{Z}\to 1.\]
Since $\phi$ and $\phi^\#$ induce the same outer automorphism of $\pi_1(X\setminus {\rm Fix}(\phi))$, 
this sequence splits if the map $\phi^\#$ acting on $X\setminus Q$
has a fixed point. However, as $\pi_1(X\setminus {\rm Fix}(\phi))$ is center free
by Claim 5, if the sequence splits then 
$\mathbb{Z}/m\mathbb{Z}$ is a subgroup of $\pi_1(\bar X)$, which is impossible as
$X\setminus {\rm Fix}(\phi)$ and hence $\bar X$ has the homotopy type of a finite CW complex
by Claim 5. We refer to \cite[Lemma 2.2]{HJ24} for more information on this line of argument.
As a conclusion, the action of $\phi^\#$ on $\pi_1(X\setminus Q)$ is fixed point free, completing the proof of Claim 6.\hfil$\blacksquare$ 

This completes the proof of \Cref{fixed}.
\end{proof} 

\begin{remark}\label{alwaysvalid}
The above proof is valid for all covers $X$ of a hyperbolic manifold $M$, branched along
a totally geodesic nullhomologous submanifold $\Sigma$ of codimension two. 
It shows that if $X$ admits a hyperbolic 
metric, then the fixed point set of an isometry of $X$ homotopic to 
an element of the deck group of $X\to M$
is diffeomorphic to the branch locus $\Sigma$, thus confirming \cite{GT87}.  
\end{remark}

\subsection{The proof of Theorem \ref{prop: no hyperbolic metric}}\label{even}
In this subsection we assume as before that 
$X$ admits a hyperbolic metric. 
Let $j$ be the involution of $X$ 
described in Fact \ref{factfixed},
and $\zeta$ be the generator of the deck group
of $X\to M$ which cyclically permutes the copies $M_{\rm cut}^1,\dots,M_{\rm cut}^d$ of $M_{\rm cut}$ in $X$. 

From now on we always denote by $F$ the component of ${\rm Fix}(j)$ containing $\Sigma$ and by
$F^\#$ the homotopic component of ${\rm Fix}(j^\#)$ whose existence was shown in Proposition 
\ref{fixed}. 
By \Cref{fixed} and Mostow rigidity for closed hyperbolic manifolds of 
dimension $n-1\geq 3$, there exists an isometry $\psi:F\to F^\#$ which maps 
$\Sigma$ to the fixed point set $\Sigma^\#$ of $\zeta^\#$. Furthermore, with a homotopy we may
identify $\Sigma$ and $\Sigma^\#$ in $X$. 
For each $i=0,\dots,d-1$, the map $(\zeta^\#)^i\circ \psi \circ \zeta^{-i}$ maps
$\zeta^{i}(F)$ isometrically onto $(\zeta^\#)^i(F^\#)$.




After possibly changing the hyperbolic metric of $X$ with an isotopy, we may assume that for each connected component $\Sigma_0$ of $\Sigma$ we have $\Sigma_0 \cap \Sigma_0^\# \neq \emptyset$, where $\Sigma_0^\#=\psi_0(\Sigma_0)$. So, for each component, we can fix a basepoint $x_0 \in \Sigma_0 \cap \Sigma_0^\#$, 
and we may assume without loss of generality that $\psi_0(x_0)=x_0$. We call such a basepoint \emph{preferred}. 
Due to \Cref{fixed}, we may also assume that
\[
    \pi_1(\Sigma_0,x_0)=\pi_1(\Sigma_0^\#,x_0)
    \quad \text{and} \quad
    \pi_1(F,x_0)=\pi_1(F^\#,x_0).
\]
In the sequel, the fundamental
group $\pi_1(X,x_0)$ will always be represented with respect to a fixed choice $x_0$ of 
preferred basepoint. 



Although by \Cref{fixed}, the cyclic group generated by $\zeta^\#$ acts freely on $X\setminus \Sigma^\#$
and the manifold $F^\#$ is homotopic to $F$, this does not necessarily imply that 
$\zeta^\#(F^\#)\cap F^\#=\Sigma^\#$. The following lemma takes care of this issue.

\begin{lem}\label{intersection}
\begin{enumerate}
    \item The differential 
    of $\zeta^\#$ acts on the normal bundle of $\Sigma^\#$ by a rotation 
    with angle $2\pi/d$. 
\item We have $F^\#\cap \zeta^\#(F^\#)=\Sigma^\#$.
\end{enumerate}
\end{lem}
\begin{proof} We begin with the proof of the second part of the lemma. 
We may assume that $\pi_1(F,x_0)=\pi_1(F^\#,x_0)$ and $\pi_1(\zeta(F),x_0)=
\pi_1(\zeta^\#(F^\#),x_0)$. Thus for a choice of lift $\tilde x_0$ of $x_0$ to the universal covering 
$\mathbb{H}^n$, limit sets of these groups in the ideal boundary $\partial \mathbb{H}^n$
and of their conjugates, acting as subgroups of the deck group,  coincide.

Let $\widetilde{ F^\#}\subset \mathbb{H}^n$  
and $\widetilde{ \zeta^\#F^\#}\subset \mathbb{H}^n$ be the (unique) lifts of $F^\#$ and $\zeta^\#( F^\#)$, respectively, which 
pass through $\tilde x_0$. Each component of $F^\#\cap \zeta^\#(F^\#)$, which is a totally geodesic embedded 
hyperplane in $F^\#$, lifts to precisely one  
$\pi_1(F^\#,x_0)$-orbit of intersections of $\widetilde{ F^\#}$ with $\pi_1(X,x_0)(\widetilde{\zeta^\#F^\#})$ 
(using the deck group action) and hence
to a $\pi_1(F^\#,x_0)$-orbit of intersections of the boundary sphere of $\widetilde{ F^\#}$ with the boundaries of the 
hyperplanes in the orbit of $\widetilde{\zeta^\#F^\#}$. These boundary spheres are precisely the limit sets of the conjugates of the 
group $\pi_1(\zeta^\#(F^\#),x_0)=\pi_1(\zeta(F),x_0)$ in $\partial \mathbb{H}^n$. 
Since $F\cap \zeta(F)=\Sigma$, the number of $\pi_1(F^\#,x_0)$-orbits of such intersection spheres is at most the number
of components of $\Sigma=\Sigma^\#$. Thus we have $F^\#\cap \zeta^\#(F^\#)=\Sigma^\#$ which completes the 
proof of the second part of the lemma. 

Let $\Sigma_0^\#$ be a component of $\Sigma^\#$. This is a totally geodesic submanifold of $X$ of 
codimension two contained in the fixed point set of $\zeta^\#$. 
Since $\zeta^\#$ is a non-trivial orientation preserving isometry of $X$ of order $d$, 
its differential acts on the normal bundle of $\Sigma_0^\#$ as a rotation with rotation angle
$2\pi k/d$ where $k$ is a generator of the cyclic group of order $d$. We have to show that 
$k=1$.

Consider again lifts $\tilde \zeta,\tilde \zeta^\#$
of $\zeta,\zeta^\#$ to the universal covering $\mathbb{H}^n$ of $X$, chosen so that 
they fix pointwise the same component $\widetilde {\Sigma_0^\#}$ of the universal covering of $\Sigma_0^\#$, which 
is a totally geodesic subspace of $\mathbb{H}^n$ of codimension two. The differential of 
$\tilde \zeta^\#$ acts on the normal bundle of $\widetilde{\Sigma^\#}$ as a rotation with rotation
angle $2\pi k/p$.

The choice of basepoint $\tilde x_0\in \widetilde{\Sigma_0^\#}$ 
determines
an identification of the unit tangent sphere $T^1_{\tilde x_0}\mathbb{H}^n$ of $\mathbb{H}^n$ at
$\tilde x_0$ with the ideal boundary $\partial \mathbb{H}^n=S^{n-1}$ of 
$\mathbb{H}^n$ by associating to a unit tangent vector $v$ the equivalence class of the geodesic ray 
with initial velocity $v$. 
The limit set
in $\partial \mathbb{H}^n=S^{n-1}$ of the stabilizer of $\widetilde{F^\#}$ 
in the deck group $\pi_1(X,x_0)$ equals the 
boundary $\partial \widetilde{F^\#}$
of $\widetilde{F^\#}$, which is an equator sphere of codimension one in $T_{\tilde x_0}\mathbb{H}^n=
\partial \mathbb{H}^n$. It contains the ideal boundary 
of $\widetilde{\Sigma_0^\#}$ as an equator sphere. 
We also know that $\zeta^\#(\partial \widetilde{F^\#})$ coincides with the limit set 
$\zeta(\partial \tilde  F)$ of the 
group 
$\pi_1(\zeta(F),x_0)$ acting on $\mathbb{H}^n$.

Now recall that $\zeta$ acts as an isometry with respect to the ${\rm CAT}(-1)$ hyperbolic cone
metric $h$ on $X$, which is quasi-isometric to the hyperbolic metric, and it acts as 
a cyclic permutation on the totally geodesic 
submanifolds $\zeta^i(F)$. Thus viewing $\partial \mathbb{H}^n$ 
as the ideal boundary of the universal covering of $X$, equipped with the hyperbolic cone metric $h$, 
we obtain that 
there is a component of $\partial \mathbb{H}^n\setminus (\partial \tilde F\cup \zeta (\partial \tilde F))$ 
which  
does not intersect any of the spheres $\zeta^i(\partial \tilde{F})$. By identifying 
$\partial \widetilde{F^\#}$ with the unit tangent sphere of $\widetilde{F^\#}$
at $\tilde x_0$, which is 
an equator sphere in $T^1_{\tilde x_0}\mathbb{H}^n$, 
and 
$\zeta^\#(\partial \widetilde{F^\#})$ with the unit tangent space of 
$\zeta^\#(F^\#)$ at $\tilde x_0$, we deduce that there is a component of 
$T_{\tilde x_0}^1\mathbb{H}^n\setminus (T^1_{\tilde x_0}F^\#\cup d\zeta^\#(T_{\tilde x_0}^1F^\#))$
not intersecting any of the spheres $d\zeta^\#(T^1_{\tilde x_0}F^\#)$ if and only if 
the differential of $\zeta^\#$ acts on the normal
bundle of $\Sigma_0^\#$ as a rotation with rotation angle $2\pi/d$. 
This completes the proof of the lemma. 
\end{proof}

\begin{remark}
The proof of the first part of Lemma \ref{intersection} relies on the analysis of
limit sets of stabilizers of preimages of the totally geodesic hyperplane $H\subset M$. It
remains valid even if 
$H$ is not fixed by an isometric involution. 
\end{remark}

With these preliminary results at hand, we can now prove \Cref{prop: no hyperbolic metric}.

\begin{proof}[Proof of \Cref{prop: no hyperbolic metric}] 
By construction, the subspace $F\cup \zeta(F)$ of $X$ separates $X$.
By the definition of the map $j$, 
the complement $X-(F\cup \zeta(F))$ contains two connected components whose closures are homeomorphic 
to $M_{\rm cut}$.
Let $Z$ be the closure of such a component. 
Its boundary consists of 
two copies of $H_0$ glued along $\Sigma$. 

By Lemma \ref{intersection}, there exists a corresponding component $M_{\rm cut}^\#$ of 
$X-(F^\#\cup \zeta^\#(F^\#))$. The boundary of its closure $Z^\#$ 
is connected and consists of two copies of $H_0$ meeting 
along $\Sigma$ with an angle $2\pi/d$. Identifying $\Sigma$ and $\Sigma^\#$ as before and choosing
a basepoint $x\in \Sigma$, we claim that $\pi_1(Z,x)=\pi_1(Z^\#,x)$. 

Namely, by \Cref{fixed}, it holds that $\pi_1(\partial Z,x)=
\pi_1(\partial Z^\#,x)$. As $\partial Z$ is a separating hypersurface in 
$X$ homotopic to $\partial Z^\#$, by the theorem of Seifert-van Kampen, we know that 
\begin{equation*}\pi_1(X,x) =\pi_1(Z,x)*_{\pi_1(\partial Z,x)} \pi_1(X\setminus Z,x)
=\pi_1(Z^\#,x)
*_{\pi_1(\partial Z^\#,x)}\pi_1(X\setminus Z^\#,x).\end{equation*}
It then follows from the normal form for amalgamated products \cite[p.186]{LS01}  
that $\pi_1(Z^\#,x)$ is isomorphic to either 
$\pi_1(Z,x)$ or to $\pi_1(X\setminus Z,x)$. 

If $d=2$ then $\pi_1(Z,x)$ is isomorphic to $\pi_1(X\setminus Z,x)$ and the claim is clear. 
If $d\geq 3$ then note that $\zeta_*=\zeta_*^\#$ maps 
$\pi_1(Z,x)$ to a proper subgroup of $\pi_1(X\setminus Z,x)$, and it maps $\pi_1(X\setminus Z,x)$
to a proper supergroup of $\pi_1(Z,x)$. Furthermore, it maps 
$\pi_1(Z^\#,x)$ to a proper subgroup of $\pi_1(X\setminus Z^\#,x)$, and it maps 
$\pi_1(X\setminus Z^\#,x)$ to a proper supergroup of 
$\pi_1(Z^\#,x)$. Thus we have $\pi_1(Z,x)=\pi_1(Z^\#,x)$ as claimed.

We argue now by contradiction and we assume that there are distinct multiples 
of $d_1\not=d_2 \in 4\bbN$ such that the cyclic $d_i$-fold branched cover $X^{(d_i)}$ admits a smooth hyperbolic 
metric for $i=1,2$. Then, for each $i=1,2$, the above discussion 
implies that there exists a hyperbolic cone manifold $M_{\rm cut}^{2\pi/d_i}$ with totally geodesic boundary $\partial M_{\rm cut}^{2\pi/d_i}$ 
homeomorphic and path isometric to $\partial M_{\rm cut}$, with singular set isometric to $\Sigma$, cone angle $2\pi/d_i$ along $\Sigma$, and $\pi_1(M_{\rm cut}^{2\pi/d_i})=\pi_1(M_{\rm cut})$.


Note that $\frac{d_1}{2}\frac{2\pi}{d_1}+\frac{d_2}{2}\frac{2\pi}{d_2}=2\pi$.
Therefore, we can glue $d_1/2$ copies of $M_{\rm cut}^{2\pi/d_1}$ and 
$d_2/2$ copies of $M_{\rm cut}^{2\pi/d_2}$ in cyclic order along the components of $\partial M_{\rm cut}^{2\pi/d_i} \setminus \Sigma$ to a smooth hyperbolic manifold $Y$.
An application of the Seifert--van Kampen theorem shows that the fundamental group 
of $Y$ is isomorphic to the fundamental group of the $(d_1+d_2)/2$-fold cyclic cover $X$ of $M$ branched along $\Sigma$. 
In particular, this fundamental group admits a finite group of automorphisms generated 
by an element $\zeta_\ast$ of order $(d_1+d_2)/2$ and an involution $j_\ast$ corresponding to the 
automorphisms induced by the homeomorphisms $\zeta$ and $j$ of $X$ (notations are as before). By the 
beginning of this proof, 
for each $i=0,\dots,(d_1+d_2)/2-1$, the fixed point group of $\zeta_\ast^i \circ j_\ast \circ \zeta_\ast^{-i}$ is the fundamental group of an embedded codimension one submanifold $F_i$ that, by construction of the hyperbolic metric on $Y$, is already totally geodesic.
Moreover, for some $i$ the totally geodesic submanifolds $F_i$ and $F_{i+1}$ intersect with angle $2\pi/d_1$, while for other $i$ they intersect with angle $2\pi/d_2$.


By Mostow rigidity, there exist isometries $\zeta^\#, j^\#$ of the hyperbolic manifold $Y$ of order $(d_1+d_2)/2$ and $2$, respectively, that induce the outer automorphism given by $\zeta_\ast$ and $j_\ast$. By \Cref{intersection}, the fixed point set of $\zeta^\#$ is a 
codimension two totally geodesic submanifold $\Sigma^\#$ freely homotopic to $\Sigma$, and thus $\Sigma^\#=\Sigma$ since $\Sigma$ is already totally geodesic in $Y$. 
Similarly, the fixed point set $(\zeta^\#)^i(F^\#)$ of the involution 
$(\zeta^\#)^i \circ j^\# \circ (\zeta^\#)^{-i}$
is a totally geodesic hyperplane freely homotopic 
to the manifold $F_i$ satisfying $\pi_1(F_i)={\rm Fix}(\zeta_\ast^i \circ j_\ast \circ \zeta_\ast^{-i})$, and thus $(\zeta^\#)^i(F^\#)=F_i$ since $F_i$ is already hyperbolic. However, as $\zeta^\#$ acts by rotation with a fixed angle in the normal bundle of $\Sigma$, the intersection angle of $(\zeta^\#)^i(F^\#)$ and $(\zeta^\#)^{i+1}(F^\#)$ is the same for all $i$. But this contradicts the fact that, by construction, the intersection angle of $F_i$ with $F_{i+1}$ varies between $2\pi/d_1$ and $2\pi/d_2$,
completing the proof of the theorem.
\end{proof}


\begin{rem}
The proof of \Cref{prop: no hyperbolic metric} for branched covers of hyperbolic manifolds of 
dimension $n\geq 4$ 
depends in a crucial way on the validity of 
Mostow rigidity for closed hyperbolic manifolds of dimension $n-1$ 
and hence is not valid for $n=3$. It also shows that at most one covering of degree $d\equiv 2$ mod 4
can admit a hyperbolic metric. 
\end{rem}

\bigskip
\noindent
MATHEMATISCHES INSTITUT DER UNIVERSIT\"AT BONN\\
ENDENICHER ALLEE 60, 53115 BONN, GERMANY \\
email: ursula@math.uni-bonn.de 

\bigskip
\noindent
MATHEMATISCHES INSTITUT DER UNIVERSIT\"AT BONN\\
ENDENICHER ALLEE 60, 53115 BONN, GERMANY \\
email: fjaeckel@math.uni-bonn.de


\begin{thebibliography}{BHW11}

\bibitem[And90]{And90}
M.~T. Anderson.
\newblock {Convergence and rigidity of manifolds under Ricci curvature bounds}.
\newblock {\em Inventiones mathematicae}, 102(2):429--446, 1990.

\bibitem[And06]{And06}
M.~T. Anderson.
\newblock {Dehn Filling and Einstein Metrics in Higher Dimensions}.
\newblock {\em Journal of Differential Geometry}, 73(2):219--261, 2006.

\bibitem[And10]{And10} M.~T. Anderson.
\newblock
{\em A survey of Einstein metrics on 4-manifolds}, 
Adv. Lect. Math. (ALM), 14, International Press, Somerville, MA 2010, 1--39.





\bibitem[Ber00]{Ber00}
N.~Bergeron.
\newblock {Premier nombre de Betti et spectre du Laplacien des certaines variétés hyperboliques}.
\newblock {\em L'Enseignement Mathématique}, 46:109--137, 2000.

\bibitem[BHW11]{BHW11}
N.~Bergeron, F.~Haglund, and D.~T.~Wise.
\newblock{Hyperplane sections in arithmetic hyperbolic manifolds}.
\newblock{\em Journal of the London Mathematical Society}, 83(2): 431--448, 2011.

\bibitem[Bes08]{Bes08}
A.~L. Besse.
\newblock{\em Einstein manifolds}. Classics in Mathematics. Reprint of the 1987 edition.
\newblock{Springer-Verlag}, 2008.

\bibitem[BCG95]{BCG95} G.~Besson, G.~Courtois, and S.~Gallot.
\newblock{Entropies and rigidit\'es des
espaces local\'ement symm\'etriques de courbure strictement n\'egative}, 
\newblock{\em Geometric and Functional Analysis}, 5:731--799, 1995. 



\bibitem[Biq00]{Biq00}
O.~Biquard.
\newblock {\em {M\'etriques d'Einstein asymptotiquement sym\'etriques}}.
\newblock Number 265 in Ast\'erisque. Soci\'et\'e math\'ematique de France,
  2000.

\bibitem[Bro82]{Bro82}
K.S.~Brown.
\newblock {\em Cohomology of groups}.
\newblock Graduate Texts in Mathematics. Springer, 1982.


\bibitem[CLW18]{CLW18} 
S. Cappell, A. Lubotzky, and S. Weinberger.
\newblock{A trichotomy for transformation groups of locally symmetric manifolds
and topological rigidity}.
\newblock{\em Advances Math.}, 327:25--46, 2018.

\bibitem[dC92]{dC92}
M.~P.~do Carmo.
\newblock{\em Riemannian Geometry}.
\newblock Birkh\"auser, 2. edition, 1992.

  
\bibitem[ES64]{ES64}
J.~Eells and J.~H.~Sampson.
\newblock{Harmonic Mappings of Riemannian Manifolds}.
\newblock{\em American Journal of Mathematics}, 86(1):106--160, 1964.

\bibitem[Em23]{Em23}
V.~Emery.
\newblock{\em Arithmetic groups}.
\newblock{Lecture notes on an introductory course on arithmetic lattices}, arXiv:2308.03752.

\bibitem[FP20]{FP20}
J.~Fine and B.~Premoselli.
\newblock {Examples of compact Einstein four-manifolds with negative
  curvature}.
\newblock {\em Journal of the American Mathematical Society}, 33(4):991--1038, 2020.




\bibitem[GT01]{GT01}
D.~Gilbarg and N.~S.~Trudinger.
\newblock {\em Elliptic Partial Differential Equations of Second Order}.
\newblock Classics in Mathematics. Springer-Verlag, 2001.
\newblock Reprint of the 1998 edition.

\bibitem[GT87]{GT87}
M.~Gromov and W.~Thurston.
\newblock{Pinching constants for hyperbolic manifolds}.
\newblock{\em Inventiones Mathematicae}, 89:1--12, 1987.

\bibitem[HJ22]{HJ22}
U.~Hamenst\"adt and F.~J\"ackel.
\newblock{\em Stability of Einstein Metrics and effective hyperbolization in large Hempel distance}.
\newblock{Preprint}, \href{https://arxiv.org/abs/2206.10438}{arXiv:2206.10438}, 2022. 


\bibitem[HJ24]{HJ24}
U.~Hamenst\"adt and F.~J\"ackel.
\newblock{Rigidity of geometric structures}.
\newblock{\em Geometriae Dedicata}, 218(1): Paper no. 16, 2024.


\bibitem[Hir56]{Hir56}
F.~Hirzebruch.
\newblock{\em Automorphe Formen und der Satz von Riemann-Roch}.
\newblock{Symposium
International de Topologia Algebraica, Mexico}, 129--144, 1956.


\bibitem[JK82]{JK82}
J.~Jost and H.~Karcher.
\newblock {Geometrische Methoden zur Gewinnung von A-Priori-Schranken f{\"u}r
  harmonische Abbildungen}.
\newblock {\em Manuscripta Mathematica}, 40:27--77, 1982.




\bibitem[KS12]{KS12} S.~Kerckhoff, and P.~Storm,
\newblock{Local rigidity of hyperbolic manifolds with geodesic boundary}.
\newblock{\em Journal of Topology}, 5:757--784, 2012.

\bibitem[Koi78]{Koi78}
N.~Koiso.
\newblock{Nondeformability of Einstein metrics}.
\newblock{\em Osaka Journal of Mathematics}, 15(2):419--433, 1978.

\bibitem[LR07]{LR07}
J.~F.~Lafont and R.~Roy.
\newblock{A note on the characteristic classes of non-positively curved manifolds}.
\newblock{\em Expositiones Mathematicae}, 25(1):21--35, 2007.

\bibitem[La73]{La73}
T.Y.~Lam.
\newblock{\em Algebraic Theory of Quadratic Forms}. 
\newblock{Benjamin Inc.}, 1973.


\bibitem[LS01]{LS01}
R.~C.~Lyndon and P.~E.~Schupp.
\newblock{\em Combinatorial Group Theory}. Classics in Mathematics. Reprint of the 1977 edition.
\newblock{Springer Verlag}, 2001.




\bibitem[Nov65]{Nov65}
S.~P. Novikov.
\newblock{Topological invariance of rational Pontrjagin classes}.
\newblock{\em Soviet Mathematics. Doklady}, 6:921--923, 1965.

\bibitem[O20]{O20} 
P.~Ontaneda. 
\newblock{\rm Riemannian hyperbolization}.
\newblock{\em  Publications mathématiques de l'IHÉS}, 131:1--72, 2020.


\bibitem[Pet16]{Pet16}
P.~Petersen.
\newblock {\em {Riemannian Geometry}}.
\newblock Springer, 3. edition, 2016.



\bibitem[Pr42]{Pr42}
A.~Preissmannn.
\newblock{Quelques propriétés globales des espaces de Riemann}.
\newblock{\em Commentarii Mathematici Helvetici}, 15:175--216, 1942.

\bibitem[Sa86]{Sa86} 
J.H.~Sampson.
\newblock{Applications of harmonic maps to K\"ahler geometry}. 
\newblock{Contemp. Math.}, 49:125--133, 1986.


\bibitem[Top06]{Top06}
P.~Topping.
\newblock {\em {Lectures on the Ricci Flow}}.
\newblock London Mathematical Society Lecture Note Series. Cambridge University
  Press, 2006.

\bibitem[W62]{W62} 
J.A.~Wolf. 
\newblock{Discrete groups, symmetric spaces, and global holonomy}. 
\newblock{\em Amer. J. Math.}, 84:527--542, 1962.



\end{thebibliography}
\end{document}